\documentclass{article}
\usepackage{fontspec}
\usepackage{graphicx} \usepackage{mathtools}
\usepackage{calrsfs} \usepackage{microtype} 
\usepackage{enumitem}
\usepackage{xcolor}
\usepackage{comment}
\usepackage{amsfonts} \usepackage{amsmath} \usepackage{amsrefs}
\usepackage{amssymb} \usepackage{amstext} \usepackage{amsthm}
\usepackage{multirow}
\usepackage{array}

\usepackage{cancel} \usepackage{caption} 
\usepackage{hyperref} \usepackage{cleveref} \usepackage{subcaption}

\theoremstyle{plain}

\newtheorem{theorem}{Theorem}[section]

\newtheorem{corollary}[theorem]{Corollary}
\newtheorem{lemma}[theorem]{Lemma}
\newtheorem{proposition}[theorem]{Proposition}

\theoremstyle{definition} \newtheorem{definition}[theorem]{Definition}
\newtheorem{remark}[theorem]{Remark}
\newtheorem{example}[theorem]{Example}

\newcommand{\cF}{\mathcal F}

 \def\cC{\mathcal C}

  \def\cF{\mathcal F}
  
\def\cH{\mathcal H}  
 
\newcommand{\tr}{\mathrm{Tr}} \newcommand{\PG}{\mathrm{PG}}
\newcommand{\AG}{\mathrm{AG}}

\newcommand{\KK}{\mathbb{K}}
\newcommand{\FF}{\mathbb{F}}
\newcommand{\cV}{\mathcal V}

\newcommand{\GF}[1]{\mathbb{F}_{#1}}

\title{Minimal codes from hypersurfaces \\
in even characteristic }
\author{Angela Aguglia, Luca Giuzzi,
Giovanni Longobardi, Viola Siconolfi }
\date{}

\begin{document}

\maketitle

\begin{abstract}
The setting of projective systems can be used to study the
parameters of a projective linear code $\cC$. This can be done by considering the intersections
of the point set $\Omega$ defined by the columns of a generating matrix
for $\cC$ with the hyperplanes of a projective space.
In particular, $\cC$ is minimal if $\Omega$ is \emph{cutting set}, i.e.,
every hyperplane is spanned by its intersection with $\Omega$.
Minimal linear codes have important applications for secret sharing schemes
and secure two-party computation.

In this article we first investigate the properties of some
algebraic hypersurfaces $\cV_{\varepsilon}^r$ related to certain quasi-Hermitian varieties of $\PG(r,q^2)$, with $q=2^e$, $e>1$ odd.
These varieties give rise to a new infinite family of  linear codes which are minimal except for $r=3$ and $e\equiv 1 \pmod 4$.

In the case $r \in \{3,4\}$, we exhibit codes having at most 6 non-zero weights whose we provide the complete list. As a byproduct, we obtain $(r+1)$-dimensional codes with just $3$ non-zero weights.
We point out that linear codes with few weights are also important in  authentication codes and association schemes.

In the last part of the paper, we consider an extension of the notion
of being cutting with respect to subspaces other than hyperplanes
and introduce the definition of  a \emph{cutting gap} in order to characterize
and measure what happens when this property is not satisfied.
Finally,  we  apply these concepts to both Hermitian varieties and the hypersurfaces $\cV_\varepsilon^r$.
\end{abstract}


\smallskip


\section{Introduction}
Linear codes derived from projective geometry have been extensively investigated due to the rich interplay between algebraic and geometric structures. The approach of constructing
generator matrices for codes by considering the coordinates of point-sets in projective
spaces has been pioneered by~\cite{MacWilliams} and~\cite{Burton} and it has proved
to be extremely fruitful ever since.
In particular, given a non-degenerate $q$-ary linear $[n,r,d]$-code $\cC$, the \emph{projective system associated with $\cC$} is defined as the (multi)set $\Omega$ of points in $\PG(r-1,q)$ corresponding to the columns of a generator matrix of $\cC$. 
Conversely, any such multiset $\Omega$ yields a code $\cC(\Omega)$ by taking the coordinates of the points in $\Omega$ as the columns of a generator matrix. The resulting code $\cC(\Omega)$ is unique up to monomial equivalence and is referred to as the \emph{code generated} or \emph{induced by $\Omega$}.

This geometric approach is particularly effective when $\Omega$ arises from algebraic or geometric structures, as it often leads to codes with desirable features, such as minimality or a small number of distinct weights relative to the code length.

A linear $[n,r,d]$-code is said to be \emph{minimal} if every non-zero codeword is minimal, i.e., uniquely determined (up to a non-zero scalar multiple) by its support. Geometrically, this corresponds to the condition that the associated projective system $\Omega$ is a \emph{cutting set}, meaning that every hyperplane of $\PG(r-1,q)$ is spanned by its intersection with $\Omega$. Minimal codes are of interest not only for their geometric significance but also for their applications in cryptography, including secret-sharing schemes and secure two-party computation, see \cite{Massey, CMP}.

In this paper, we construct new families of linear codes with few weights, derived from hypersurfaces associated with quasi-Hermitian varieties in even characteristic. These constructions are noteworthy for their geometric novelty. Specifically, the hypersurfaces considered are linked to Buekenhout–Tits quasi-Hermitian varieties (or BT  quasi-Hermitian varieties, for short), which exist in every $\PG(r-1,q^2)$ with $r > 2$ and $q = 2^e$, where $e \geq 3$ is odd. The structure of these varieties enables the exploration of new geometric configurations in finite projective spaces.

In certain cases, the associated codes are minimal; in others, they exhibit exactly three distinct weights. Codes with few weights are particularly valuable in the design of authentication codes and association schemes (see \cite{CG}), and their weight distribution plays a crucial role in evaluating decoding performance and efficiency.

In the final part of the paper, we introduce a new concept, the \emph{cutting gap}, to quantify the extent to which a set $\Omega$ fails to satisfy the cutting property (i.e. that a subspace is generated by its intersection with $\Omega$). 
This notion aims to quantify how much a given projective system $\Omega$ deviates from being a strong blocking set (see~\cite{DGMP}) and also aims to open the way to extending the concept of minimality to higher Hamming weights.

%This notion provides a refined geometric invariant to capture the deviation from the cutting condition  beyond hyperplanes (see \cite{...}), and it opens the way to generalizing the concept of minimality to higher Hamming weights.
\begin{comment}
\textcolor{red}{The cutting gap serves a dual purpose: it enables a more nuanced investigation of the geometry of projective systems and lays the foundation for a broader theory of minimality in coding theory. This perspective is particularly relevant when studying codes arising from algebraic varieties, where interactions with subspaces of varying dimensions reveal deeper structural properties.}\footnote{Sembra ridondante.}
\end{comment}
The paper is organized as follows. In Section~\ref{prelim}, we recall foundational results on projective codes and quasi-Hermitian varieties, with particular emphasis on BT quasi-Hermitian varieties. Section~\ref{minlinq} is devoted to the analysis of hypersurfaces $\cV^r_{\varepsilon}$ associated with these varieties, and we determine the conditions under which they yield minimal codes. For $r = 3$, we compute the nonzero weights  of $\cC_{\varepsilon}^3 := \cC(\cV^3_{\varepsilon})$, while for $r \geq 3$, we construct a new family of three-weight codes.

In Section~\ref{cutgap}, we formally define the concept of cutting gap and explore its implications. We apply this framework to Hermitian varieties and study the intersections of $\cV^4_{\varepsilon}$ with planes and lines in $\PG(4,q^2)$, including an analysis of Fermat curves. As a result, we provide the complete sequence of cutting gaps for $\cV^4_{\varepsilon}$ and the full list of weights for the code $\cC_{\varepsilon}^4 := \cC(\cV^4_{\varepsilon})$.
In both cases, $r = 3$ or $r = 4$, the limit
\[
\lim_{q \to \infty} \frac{N(q)}{d(q)} = 1,
\]
where $N(q)$ denotes the code length and $d(q)$ the minimum distance. This implies that the asymptotic minimum distance of the constructed codes is relatively good.
In the following table, we summarize all the results obtained concerning the code $\cC_{\varepsilon}^r$. 

\begin{center}
\begin{tabular}{ |c|c| } 
\hline
\textbf{Code} & \textbf{Result}  \\
\hline
\multirow{4}{5em}{
\hspace{1.5em}$\mathcal{C}^3_{\varepsilon}$} & Theorem \ref{th:V3n}: length \\
& Theorem \ref{lemma:weights}: weights for $e\equiv 3 \pmod 4$   \\ 
& Theorem \ref{lemma:weights2}: weights  for $e\equiv 1\ \pmod 4$   \\ 
& Theorem \ref{thm:minimality}: minimality for $e\equiv 3 \pmod 4$    \\ 
& Corollary \ref{coro:dist}: $m$-th higher minimum weights  \rule{0pt}{2.5ex}  \\ 
\hline
\multirow{1}{5em}
{
$\mathcal{C}^r_{\varepsilon},r\geq 4$} & Theorem \ref{th:code4} : minimality for $e$ odd \rule{0pt}{2.5ex} \\ 
\hline
{$\mathcal{C}(\overline{\cV_{\varepsilon}^r})$} &   Theorem \ref{thm:codeVbar}: weights
\rule{0pt}{2.5ex} \\ 
\hline
$\mathcal{C}^4_{\varepsilon}$ &   Proposition \ref{prop:code4}:\,minimality, length and weights
\rule{0pt}{2.5ex} \\
\hline
\end{tabular}
\end{center}

\section{Preliminaries}
\label{prelim}
All codes studied in this paper are defined in the Hamming metric. 
In order to make the article more readable, we shall briefly recall in this section some classical notions from coding theory. For more details, the reader can refer to \cites{HP,TsVla}.
We shall also recall here the construction of BT quasi-Hermitian varieties, to
be used in the remainder of the work.

\subsection{Coding Theory}

Let $\cC$ be a $q$-ary linear $[n,r,d]$-code, i.e. a $r$-dimensional linear
$\GF{q}$-subspace of $\GF{q}^n$
 with minimum Hamming distance $d$. The \emph{support} of a codeword
$\boldsymbol{c}=(c_1,\dots,c_n)\in\cC$ is the set $\mathrm{supp}(\boldsymbol{c}):=\{i: c_i\neq 0\}$. The support $\mathrm{supp}(\cC)$ of a code $\cC$ is the union of the supports of its codewords. The code $\cC$ is called \textit{non-degenerate} if $\mathrm{supp}(\cC) = \{1,\ldots, n\}$; degenerate otherwise.
The \textit{dual} of an $[n,r,d]$-linear code $\cC$ is defined as 
$$\cC^{\perp} = \{\boldsymbol{c}' \in \GF{q}^n \colon \boldsymbol{c} \cdot \boldsymbol{c}' = 0\quad \forall \boldsymbol{c} \in \cC\},$$
where $\cdot$ is the standard dot product of $\GF{q}^n$.
A code is non-degenerate if its dual has minimum distance at least $2$.

\begin{definition}
A codeword $\boldsymbol{c}\in\cC$ is \emph{minimal} if for any non-zero codeword $\boldsymbol{c'}\in\cC$ with
$\mathrm{supp}(\boldsymbol{c'})\subseteq \mathrm{supp}({\boldsymbol{c}})$ we
have $\boldsymbol{c'}\in\langle\boldsymbol{c}\rangle$.
A code $\cC$ is \emph{minimal} if all of its non-zero codewords are
minimal.
\end{definition}

A code $\cC$ is \emph{projective} if its dual has minimum distance at least $3$. We observe that a $q$-ary $[n,r,d]$  code $\cC$ is projective if
and only if the columns of any of
its generator matrices are pairwise non-proportional, see~\cite{CK}. In particular, they
can be regarded as the coordinates of a set $\Omega$ of
$n$ distinct points in $\PG(r-1,q)$. We shall call $\Omega$ the \emph{projective system}
of $\cC$. Conversely, given a set $\Omega$ of $n$ distinct points in $\PG(r-1,q)$
we can define a linear $[n,r]$-code $\cC:=\cC(\Omega)$ whose generator
matrix has as columns the coordinates of the points of $\Omega$. Clearly,
$\cC(\Omega)$ is not uniquely defined as a code, but it is none-the-less unique up
to monomial equivalence.

We recall the following result from from \cite{ABN} and \cite{TQLZ}:
\begin{theorem}\label{thm:min_code}
Let $\Omega$ be a set of $N$ points of $\PG(r-1, q)$ such that $\langle \Omega \rangle = \PG(r-1, q)$. For
each $i \in \{1,\ldots, N\}$, let $\boldsymbol{\textnormal{v}}_i \in \GF{q}^{r}$, 
 be a fixed representative of a point $P_i= \langle \textbf{\textnormal{v}}_i \rangle \in \Omega$ and
denote by $\cC(\Omega)$ the projective linear code having generator matrix whose columns are
the vectors $\textbf{\textnormal{v}}_i$. Then $\cC(\Omega)$ is a minimal code if and only if for any hyperplane $\Pi$ of
$\PG(r-1, q)$,
\[
\langle\Pi \cap \Omega \rangle  = \Pi, 
\]
that is $\Omega$ is \emph{cutting}.
\end{theorem}

Ashikhmin and Barg in \cite{AB} provided a sufficient
condition so that an $[n,r]$-linear
code $\cC$ over $\GF{q}$ is minimal, namely
\begin{equation}\label{suff}
 \frac{w_{\min}}{w_{\max}}>\frac{q-1}{q}, \end{equation}
where $w_{\min}$ and $w_{\max}$ are respectively the minimum and
maximum weight of non-zero codewords of $\cC$.

The intersection numbers of a projective system with the hyperplanes is
closely related
to the Hamming weights of a code. In particular the minimum distance of a code $\cC=\cC(\Omega)$, with $\Omega$ a
projective system, is
\[
d= |\Omega|-\max\{|\Omega\cap\Pi|: \Pi \text{ hyperplane of } \PG(\langle\Omega\rangle)\}.
\]
Generalized Hamming weights for linear codes have been introduced by Wei in~\cite{W91} as a
way to characterize the performance for a particular type of a linear code.
\begin{definition}
Let $\cC$ be a code and ${\mathcal D}$ a subcode of $\cC$.
The \emph{$k$-th generalized Hamming weight} of $\cC$ is
\[ d_{k}(\cC):=\min\{ |\mathrm{supp}({\mathcal D})|\colon {\mathcal D}\leq\cC,
\dim(\mathcal{D})=k \}.\]
\end{definition}
Clearly, $d_1(\cC)$ corresponds to the usual minimum Hamming distance $d$ of $\cC$.
More in general, it can be seen that for $1\leq k\leq r-1$ the
$k$--th Hamming weight of a projective code $\cC=\cC(\Omega)$ is given by 
\[
d_k(\cC)=|\Omega|-\max\{|\Omega\cap\Pi|: \mathrm{codim}(\Pi)=k \}.
\]

%Given a set of points $\cH$ and $C(\cH)$ as in the statement of \Cref{thm:min_code}, we define the $k$-higher weights of $C(\cH)$ as:
%\[
%d_k(C(\cH)):=|\cH|-%\max\{|\cH\cap \pi|: \pi %\text{ is a $k$-codimensional %subspace of } \PG(r,q) \}.
%\]

%If $\pi$ us a subspace of $\PG(r,q)$ we define the weight of the elements in $C(\cH)$ associated to $\pi$ hyperplane  as
%\[
%d_{\pi}(\cH):=|\cH|-|\cH\cap \pi|.
%\]

%codice q- divisibile
\subsection{Quasi-Hermitian varieties}\label{subsec:QHvar}
Quasi-Hermitian varieties in $\mathrm{PG}(r, q^2)$ are point sets having the same intersection numbers with respect to hyperplanes as a non-singular Hermitian variety $\mathcal{H}(r, q^2)$ of $\mathrm{PG}(r, q^2)$. Notice that a point set $\mathcal{S}$ in $\mathrm{PG}(r,q^2)$, with $r > 2$, having the same intersection numbers with respect to hyperplanes as a non-singular Hermitian variety also has the same number of points as $\mathcal{H}(r, q^2)$ (see \cite{S22}).

These varieties can be seen as a natural extension to higher dimensions of the notion of a unital in the projective plane.

A \emph{unital embedded} in $\mathrm{PG}(2,q^2)$ is a set of $q^3+1$ points such that each line intersects it in either $1$ or $q+1$ points. The Hermitian curve $\mathcal{H}(2,q^2)$ is a unital, known as the \emph{classical unital}.

More precisely, the cardinality of a quasi-Hermitian variety in $\mathrm{PG}(r,q^2)$ is
\[
| \mathcal{H}(r, q^2) | = \frac{(q^{r+1} + (-1)^r)(q^r - (-1)^r)}{q^2 - 1},
\]
while its intersection numbers with respect to hyperplanes are respectively
\[
| \mathcal{H}(r-1, q^2) | = \frac{(q^r + (-1)^{r-1})(q^{r-1} - (-1)^{r-1})}{q^2 - 1},
\]
and
\[
|\Pi_0 \cap \mathcal{H}(r-2, q^2)| = \frac{(q^r + (-1)^{r-1})(q^{r-1} - (-1)^{r-1})}{q^2 - 1} + (-1)^{r-1} q^{r-1},
\]
where $\Pi_0$ denotes a hyperplane tangent to the variety.

\begin{comment}

Quasi-Hermitian varieties in $\PG(r, q^2)$ are point
sets having the same intersection numbers with respect to hyperplanes
as a non-singular Hermitian variety $\cH(r, q^2)$ of $\PG(r, q^2)$. Notice that a point set $\mathcal{S}$ of $\PG(r,q^2)$, $r>2$, having the same intersection numbers with respect to hyperplanes as a non‐singular Hermitian variety has also the same number of points as $\cH(r, q^2)$ (see \cite{S22}).
They are an extension to higher dimensions of the notion of
unital in the projective plane.

A \emph{unital embedded} in $\mathrm{PG}(2,q^2)$  is a set  of $q^3+1$ points such that each line intersects it in either  $1$ or $q+1$ points. The Hermitian curve $\cH(2,q^2)$ is a unital, called the \emph{classical unital}.

More in detail, the cardinality of a
quasi-Hermitian variety of $\PG(r,q^2)$ is
\[| \cH(r, q^2) |= \frac{(q^{r+1} + (-1)^r)(q^r-(-1)^r)}
{(q^2-1)}, \]
while its intersection numbers with respect to hyperplanes are respectively
\[| \cH(r-1, q^2) |= \frac{(q^r + (-1)^{r-1})(q^{r-1}-(-1)^{r-1})}{
q^2-1} ,\]
and 
\[|\Pi_0 \cH(r-2, q^2)|= \frac{(q^r + (-1)^{r-1})(q^{r-1}-(-1)^{r-1})}{q^2-1} +
(-1)^{r-1}q^{r-1}.\]
Here $\Pi_i$ denotes an $i$-dimensional space of $\PG(r, q^2)$ and $\Pi_0\cH(r-2, q^2)$ is a cone, the join of
the vertex $\Pi_0$ to a non-singular Hermitian variety $\cH(r-2, q^2)$ in a projective subspace $\Pi_{r-2}$
which does not contain $\Pi_0$.

\end{comment}

A non-singular Hermitian variety of $\PG(r, q^2)$ is, by definition,
a quasi-Hermitian variety: \emph{the classical quasi-Hermitian variety}.

Angela Aguglia  introduced in~\cite{A} a new  family of  non-classical quasi-Hermitian varieties which are linked to Buekenhout-Tits unitals.
We shall call these varieties BT quasi hermitian varieties. They can
be constructed as follows.

Let $e$ be an odd number, $r\geq2$ and consider the projective space $\PG(r,q^2)$ of order $q^2=2^{2e}$ and dimension $r$. Denote by $$(X_0, X_1, \ldots, X_r),$$ the homogeneous coordinates of its points. Fix the hyperplane $$\Pi_{\infty}:=\{(X_0, X_1,\ldots, X_r)\in \PG(r,q^2)| \, X_0=0 \}$$ as the hyperplane of the `at infinity' and denote by $\AG(r,q^2)$  the affine space  $\PG(r,q^2)\setminus \Pi_{\infty}$ whose points have coordinates $(x_1,\ldots,x_r)$ where $x_i=X_i/X_0$ for $i=1,\ldots,r$. 

Write
\[
\tr_{q^2/q}(x):=x+x^q,\qquad \mathrm{N}_{q^2/q}(x):=x^{q+1}\qquad \text{ for all }x\in\GF{q^2},
\]
for respectively the trace and norm functions of $\GF{q^2}$ over $\GF{q}$.\\
Choose now $\varepsilon\in \GF{q^2}\setminus \GF{q}$ such that there exists a $\delta \in \GF{q}\setminus \{1\}$, $\tr_{q/2}(\delta)=1$ that satisfies $\varepsilon^2+\varepsilon+\delta=0$. Then, it is straightforward to see that $\tr_{q^2/q}
     (\varepsilon)=1$.\\
%the following equalities hold:
%\begin{enumerate}[label=\roman*)]
  %   \item   $\varepsilon^{2q}+\varepsilon^q+\delta=0$;
   %  \item   $\tr_{q^2/q}(\varepsilon)^2=\tr_{q^2/q}%(\varepsilon)$;
     %\item  
    % \item   {$\mathrm{N}_{q^2/q}(\varepsilon)=\delta$}.
     %\end{enumerate}
If $q=2^e$ with $e$ an odd number, then the function $\sigma:x\rightarrow x^{2^{\frac{e+1}{2}}}$ is an automorphism of $\GF{q}$.

Let now $\cV^r_{\varepsilon}$ be the projective variety
of $\PG(r,q^2)$ with affine equation
\begin{equation}\label{eq:V}
x_r^q+x_r=\Gamma_{\varepsilon}(x_1)+\ldots+\Gamma_{\varepsilon}(x_{r-1})
\end{equation}
where \[ \Gamma_{\varepsilon}(x)=[x+(x^q+x)\varepsilon]^{\sigma+2}+(x^q+x)^{\sigma}+(x^{2q}+x^2)\varepsilon+x^{q+1}+x^2. \]
Let also $\cH^r_\varepsilon$ be the variety  whose affine points
are the same as those of $\cV^r_{\varepsilon}$, but whose 
points at infinity are given by the Hermitian cone
\begin{equation}\label{eq:Hinf}
\cH^r_{\varepsilon,\infty}:=\{(0,X_1,\ldots,X_r)|\quad X_1^{q+1}+\ldots+X_{r-1}^{q+1}=0\}.
\end{equation}
It is shown in~\cite{A} that $\cH^r_\varepsilon$ is a non-classical quasi-Hermitian variety of $\PG(r,q^2)$ which, for $r=2$, is a Buekenhout-Tits unital.
We call $\cH^r_\varepsilon$ a \emph{Buekenhout-Tits} or \emph{BT quasi-Hermitian variety} of $\PG(r,q^2)$ for $r\geq2$.

Note that, since $\cH^r_\varepsilon$ is a quasi-Hermitian variety and it has the same affine points of $\cV^r_{\varepsilon}$, every hyperplane $\pi\neq \Pi_{\infty}$ in $\PG(r,q^2)$ meets $\cV^r_\varepsilon$ in at least one affine point.
\begin{theorem}\label{thm:cardinality}
The number of affine points  of $\cV^r_{\varepsilon}$ is $q^{2r-1}$.
\end{theorem}
{\begin{proof}
It follows trivially from the fact that $\cH_\varepsilon$ is a quasi-Hermitian variety,  $\cH^r_{\varepsilon,\infty}$ is a cone over $\cH(r-2,q^2)$ and $|\cV^r_{\varepsilon} \cap \AG(r,q^2)|=|\cH^r_{\varepsilon} \cap \AG(r,q^2)|$.
\end{proof}}

\subsection{Bounds on the number of points of hypersurfaces}
Let $\theta_q(r)$ be the number of points of the $r$-dimensional finite projective space $\PG(r, q)$, that is
$\theta_q(r)=\frac{q^{r+1}-1}{q-1}=q^r +q^{r-1}+\ldots+1$ for $r\geq0$  and $\theta_q(-1)=0$.
We will use the following result from~\cite{HK}.
\begin{theorem}\label{RationalVar1}
Let $\mathcal X$ be a hypersurface of degree $\nu$  in $\PG(r, q)$ without an $\GF{q}$-linear component and let  
 $N_q(\mathcal X)$ be the number of $\GF{q}$-rational points of $\mathcal X$.  Then
 \begin{equation}
 N_q(\mathcal X) \leq(\nu-1)q^{r-1} +\nu q^{r-2} +\theta_q(r-3).
 \end{equation}
\end{theorem}

\section{A family of minimal linear $q$-ary codes, $q$ even}
\label{minlinq}
Let $H$ be a quasi-Hermitian variety of $\PG(r, q^2)$. As shown in~\cite{ACG}, the linear code $\cC(H)$ generated by $H$ is minimal. This was established by applying Theorem \ref{thm:min_code}, since the sufficient condition for minimality given by Ashikhmin and Barg \eqref{suff} could not be used.

Consider the BT quasi-Hermitian variety $ \cH_{\varepsilon}^r = (\mathcal{V}_{\varepsilon}^r \setminus \Pi_{\infty}) \cup \mathcal{\cH_{\varepsilon,\infty}^r} $. We ask whether $\mathcal{V}_{\varepsilon}^r $ induces minimal codes.

\subsection{The variety $\cV_{\varepsilon}^3$}

We first study the projective 
linear code $\cC^3_{\varepsilon}=\cC(\cV^3_\varepsilon)$. 
%induced by the projective system of the $\GF{q^2}$-rational points of
%$\cV^3_{\varepsilon} \subset \PG(3,q^2)$.
As an initial step, we investigate the variety at infinity, defined as \( \mathcal{V}^3_{\varepsilon,\infty} := \mathcal{V}^3_{\varepsilon} \cap \Pi_{\infty} \).

\begin{lemma}\label{lemma:intersection}
Let $\Pi_{\infty}$ be the plane at infinity of $\PG(3,q^2)$. Then,
\[
\cV^3_{\varepsilon,\infty}=\begin{cases}
 \ell_0 & \text{ if } e\equiv 1\; mod\, 4\\
    \ell_1\cup \ell_2 \cup \ell_3 & \text{ if }
    e\equiv 3 \;mod \,4,
\end{cases}
\]
where $\ell_0:X_0= X_1+X_2=0$ and $\ell_i:X_0=X_1+c_iX_2=0$ for $i=1,2,3$ and
$c_1$, $c_2$ and $c_3$ are the $(2^\frac{e-1}{2}+1)$-th  roots of unity in $\GF{q^2}$.
\end{lemma}
\begin{proof}
    The points in $\cV^3_{\varepsilon,\infty}$, satisfy equation
    \[
X_1^{\sigma+2}+X_2^{\sigma+2}=0.
    \]
Recalling that $\sigma: x  \longmapsto x^{2^{\frac{e+1}{2}}}$, we obtain 
    \[
   X_1^{2^{\frac{e+1}{2}}+2}+X_2^{2^{\frac{e+1}{2}}+2}= X_1^{2(2^{\frac{e-1}{2}}+1)}+X_2^{2(2^{\frac{e-1}{2}}+1)}=0
    \]
and, hence,
    \[
X_1^{2^{\frac{e-1}{2}}+1} +X_2^{2^{\frac{e-1}{2}}+1}=0.
    \]
    An easy non-trivial solution to this equation is the point $P_{\infty}(0,0,0,1)$. In order to find other solutions, we assume $X_2\neq 0$ and study
    \[\left (
\frac{X_1}{X_2}\right )^{2^{\frac{e-1}{2}}+1}=1.
    \]
    This reduces to the study of the $(2^{\frac{e-1}{2}+1})$-th roots of unity in $\GF{q^2}$ and, hence, to compute the $\gcd(2^{\frac{e-1}{2}}+1, q^2-1)$.
Let $d=\gcd \left (\frac{e-1}{2},2e \right )$. Since \[\gcd \left (2^\frac{e-1}{2}+1,2^{2e}-1 \right )=\begin{cases}
    1 &  \text{if $\frac{2e}{d}$ is odd} \\
    2^d+1 & \text{otherwise}
\end{cases}
\]
and $e$ is odd, then it is enough to compute $d$.\\
If $e \equiv 1 \pmod 4$, then $e=4m+1$ for some $m \in \mathbb{N}$. So, $d=\gcd(2m,2e)=2\gcd(m,e)$, implying that $2e/d$ is odd. If $e \equiv 3 \pmod 4$, then $e=4m+3$ for some $m \in \mathbb{N}$ and $$d=\gcd(2m+1,2(4m+3))=\gcd(2m+1,4m+3)=1.$$
This completes the proof.
\begin{comment}
    Let $p$ an odd prime number such that $p|GCD(2^{\frac{e-1}{2}}+1, q^2-1)$, then we have
    \[
   q^2=2^{2e}\equiv  1 \;mod \, p 
    \]
    and 
    \[
    2^{\frac{e-1}{2}}\equiv -1 \;mod \,p \Rightarrow 2^{e-1}\equiv  1 \;mod\, p \iff 2^e\equiv  2 \;mod\, p \iff 2^{2e}\equiv  4 \;mod \,p
    \]
    we obtain an absurd unless $4\equiv 1\; mod \, p$,
    so the only possible prime number $p$ such that $p|GCD(2^{\frac{e-1}{2}}+1, q^2-1)$ is $3$. In particular we see that $q^{2}-1$ is always a multiple of $3$. We study 
    \[
    2^{\frac{e-1}{2}}+1\equiv 0 \;mod \,3 \iff \frac{e-1}{2}\equiv  1 \; mod\, 2 \iff e\equiv  3 \;mod\, 4.
    \]
The thesis follows.   
\end{comment}
\end{proof}

By  \Cref{thm:cardinality},  we have that $| \cV^3_{\varepsilon} \setminus \Pi_\infty|=q^{5}$. So, by \Cref{lemma:intersection}, we get the following.
\begin{theorem}
  \label{th:V3n}
    \begin{equation}\label{formula:n}
   |\cV^3_{\varepsilon}|=
   \begin{cases}
       q^{5}+q^2+1 & if\quad   e\equiv 1 \pmod 4 \\
      q^{5}+ 3q^2+1 & if \quad  e\equiv 3 \pmod 4. 
   \end{cases}
 \end{equation}
\end{theorem}

\begin{theorem}\label{lemma:weights}
If $e \equiv 3 \pmod 4$, then the weights of the projective code $\cC^3_{\varepsilon}:=\cC(\cV^3_{\varepsilon})$
  belong to the set
\begin{multline*}
\large\{q^5,q^5-q^3+3q^2, q^5-q^3+2q^2, 
q^5-q^3+3q^2+q-2, \\ q^5-q^3+2q^2+q-2 \large\}
.\end{multline*}
\end{theorem}

{\begin{proof}
We have to study the cardinality of the intersections between $\cV^3_{\varepsilon}$ and a plane $\pi: aX_0+bX_1+cX_2+dX_3=0$ of $\PG(3,q^2)$.
\noindent
If $(b,c,d)=(0,0,0)$ then $\pi=\Pi_{\infty}$ and by \Cref{lemma:intersection}, $|\cV^3_{\varepsilon,\infty}|=3q^2+1$.

Next, we distinguish the following cases:
\begin{enumerate}[label=\alph*)]
        \item\label{tA} $d=0$, $b \neq 0\neq c$ and $\frac{b}{c} \in \{c_1,c_2,c_3\}$, where $c_1,c_2,c_3$ are the $(2^{\frac{e-1}{2}}+1)$-th roots of unity;
         \item\label{tB} $d=0$,  $(b,c)\neq (0,0)$ and  $\frac{b}{c}$  is not a $(2^{\frac{e-1}{2}}+1)$-th root of unity;
        \item\label{tC}  $d\neq 0$.
        \end{enumerate}
If $\pi$ is of type \ref{tA} then  $|\pi\cap \cV^3_{\varepsilon,\infty}|=q^2+1$  whereas if it is of type \ref{tB} then  $\pi\cap \cV^3_{\varepsilon,\infty}$ consists only of the point $P_\infty(0,0,0,1)$.

 In order to determine the number of the affine points of $\cV^3_{\varepsilon} \cap \pi$  in cases \ref{tA} or \ref{tB} we can assume without loss of generality $b=1$.  The case of $b=0$ and $c\neq 0$ follows immediately by swapping the roles of $X_1$ and $X_2$. 
So we need to compute the number of solutions for the following system in unknowns  $(x,y,z) \in \GF{q^2}^3$
    \begin{equation}\label{eq:system}
    \begin{cases}
z^q+z=\Gamma_{\varepsilon}(x)+\Gamma_{\varepsilon}(y)\\
x+cy+a=0.
    \end{cases}
    \end{equation}
    In order to solve the system above, let consider the $\GF{q}$-basis $\{1,\varepsilon\}$ of $\GF{q^2}$ introduced in Section~\ref{subsec:QHvar}.
    Then,
    \[
x=x_0+\varepsilon x_1\;\; y=y_0+\varepsilon y_1\;\; z=z_0+\varepsilon z_1\;\;
a=a_0+\varepsilon a_1\;\; c=c_0+\varepsilon c_1.
    \]
with $x_i,y_i,z_i,a_i,c_i \in \GF{q}$, $i=1,2$.
Note that $x+x^q=x_1$, \begin{equation}
\begin{split}
\Gamma_\varepsilon(x)&=[(x_0+\varepsilon x_1)  + x_1 \varepsilon ]^{\sigma+2}+x_1^{\sigma}+x_1^2 \varepsilon + (x_0+x_1 \varepsilon)^{q+1}+x_0^2+\varepsilon^2 x_1^2\\
&=x_0^{\sigma+2}+x_1^\sigma+x_1^2 \varepsilon +x_0^2 + x_0x_1 +\varepsilon^{q+1}x_1^2+x_0^2+x_0^2+\varepsilon^2 x_1^2\\
&=x_0^{\sigma+2}+x_1^\sigma+x_0x_1+x_1^2(\varepsilon^2+\varepsilon+\delta)\\
&=x_0^{\sigma+2}+x_0x_1+x_1^\sigma.
\end{split}
\end{equation}
and, similarly, $\Gamma_{\varepsilon}(y)=y_0^{\sigma+2}+y_0y_1+y_1^\sigma.$\\
Substituting the expressions obtained above in System \eqref{eq:system}, we obtain
   \begin{equation}\label{eq:system2}
    \begin{cases}
z_1=x_0^{\sigma+2}+x_0x_1+x_1^\sigma+y_0^{\sigma+2}+y_0y_1+y_1^\sigma\\
x_0+\varepsilon x_1=(c_0y_0+\delta c_1y_1+a_0)+\varepsilon( c_1y_0+c_0y_1+c_1y_1+a_1)
    \end{cases}
    \end{equation}
By the second equation in \eqref{eq:system2}, we derive
\begin{equation}\label{eq:equalities}
x_0=c_0y_0+\delta c_1y_1+a_0 \text{ and } x_1=c_1y_0+c_0y_1+c_1y_1+a_1.
\end{equation}
We see that the value of $x$
is uniquely determined by $y$.
Substituting \eqref{eq:equalities} in the first equation of
\eqref{eq:system}, we see that $z_1$ is also uniquely determined
by $y$,
while $z_0$ ranges in $\GF{q}$.
Since $y$ ranges in $\GF{q^2}$, we get that system
\eqref{eq:system} has $q^3$ solutions; thus
the cardinality of $\cV^3_{\varepsilon} \cap\pi$ is respectively $q^3+q^2+1$ in case \ref{tA} or $q^3+1$ in case \ref{tB}.

We now study
 $|\pi\cap \cV^3_{\varepsilon}|$ in case \ref{tC}.
 Recall that $|\pi\cap \cV^3_{\varepsilon}\cap \AG(3,q^2)|=|\pi\cap \cH^3_{\varepsilon}\cap \AG(3,q^2)|$.\\
 Since $\cH^3_{\varepsilon}$ is a quasi Hermitian variety,
 the cardinality of the intersection between $\cH^3_{\varepsilon}$ and a generic plane in $\PG(3,q^2)$ is either $q^3+1$ or $q^3+q^2+1$; so, we have
$$\{q^3+1,q^3+q^2+1\} \ni |\pi\cap \cH^3_{\varepsilon}|=|\pi\cap \cH^3_{\varepsilon}\cap \AG(3,q^2)|+|\pi\cap \cH^3_{\varepsilon}\cap \Pi_{\infty}|.$$
Recalling that the set of infinity points in $\cH^3_{\varepsilon} $ is given by $q+1$ lines passing through $P_{\infty}(0,0,0,1)$ and $\pi$ does not contain $P_{\infty}$, we get
$$|\pi\cap \cV^3_{\varepsilon}\cap \AG(3,q^2)|=|\pi\cap \cH^3_{\varepsilon}\cap \AG(3,q^2)| \in \{ q^3-q,q^3+q^2-q\}.$$
So, if $\pi$ is of type \ref{tC}, $|\pi\cap \cV^3_{\varepsilon}| \in \{q^3-q+3,q^3+q^2-q+3\}$.
Finally, by  \eqref{formula:n}, the weights of $\cC^3_{\varepsilon}$ with respect to planes of $\PG(3,q^2)$ are as in the statement, concluding the proof.
\end{proof}
}

{\begin{theorem}\label{lemma:weights2}
If $e \equiv 1 \pmod 4 $, then the weights of $\cC^3_{\varepsilon}=\cC(\cV^3_{\varepsilon})$ belong to the set
\[
\large\{q^5,q^5-q^3, q^5-q^3+q^2, 
q^5-q^3+q,q^5-q^3+q^2+q
\large\}.
\]
\end{theorem}
\begin{proof}
    Since $\cV^3_{\varepsilon,\infty}:=\Pi_{\infty}\cap \cV^3_{\varepsilon}:X_1+X_2=X_0=0$, the weight for the plane at infinity is  $|\cV^3_{\varepsilon}|-q^2-1$.
    We split the set of affine projective planes 
 with equations $aX_0+bX_1+cX_2+dX_3=0$, $(b,c,d) \in \GF{q^2}^3 \setminus \{(0,0,0)\}$ of $\PG(3,q^2)$ in three subsets:
    \begin{enumerate}[label=\alph*)]
        \item\label{tAA}  $d=0$ and $b=c$, then $\pi: aX_0+X_1+X_2=0$ and  $|\pi\cap \cV^3_{\varepsilon,\infty}|=q^2+1$; 
        \item\label{tBB} $d=0$ and $b\neq c$, then $\pi\cap \cV^3_{\varepsilon,\infty}=P_{\infty}(0,0,0,1)$;
        \item\label{tCC} $d\neq 0$, then $\pi\cap \cV^3_{\varepsilon,\infty}=Q$ where $Q\neq P_{\infty}$.
    \end{enumerate}
    Now we compute the cardinalities of the intersections between $\cV^3_{\varepsilon}$ and the planes $\pi$ mentioned above. For the plane of type \ref{tAA} the cardinality is $q^3+q^2+1$ and, consequently the weight is $|\cV^3_{\varepsilon}|-q^3-q^2-1$; for the planes of type \ref{tBB} the
    cardinality is $q^3+1$ and the weight is $|\cV^3_{\varepsilon}|-q^3-1$.
    This is obtained studying the number of solutions of  System~\eqref{eq:system}  in the proof of \Cref{lemma:weights}.
Finally, we tackle the study of $|\pi\cap \cV^3_{\varepsilon}|$ for a plane of type~\ref{tCC}. Then, since $\cV^3_{\varepsilon}\cap \AG(3,q^2)=\cH^3_{\varepsilon} \cap \AG(3,q^2)$ and $|\pi \cap \cH^3_{\varepsilon}|\in \{q^3+1,q^3+q^2+1\}$, we have that
\begin{equation}
\begin{split}
    |\pi  \cap \cV^3_{\varepsilon}|&=  |\pi  \cap \cV^3_{\varepsilon} \cap  \AG(3,q^2)|+   |\pi  \cap \cV^3_{\varepsilon,\infty}|\\
    &=
    |\pi  \cap \cH^3_{\varepsilon} \cap  \AG(3,q^2)|+1=   |\pi  \cap \cH^3_{\varepsilon}|-|\pi  \cap \cH^3_{\varepsilon} \cap \Pi_{\infty}|+1\\
    &= |\pi  \cap \cH^3_{\varepsilon}|-q \in \{q^3-q+1,q^3+q^2-q+1\}.
    \end{split}
\end{equation}
So, the weight of a plane of type \ref{tCC}) is either  $|\cV^3_{\varepsilon}|-(q^3+q^2-q+1)$ or $|\cV^3_{\varepsilon}|-(q^3-q+1)$. This concludes the proof.
\end{proof}}

\begin{theorem}\label{thm:minimality}
    The code $\cC^3_{\varepsilon}=\cC(\cV^3_{\varepsilon})$ is a minimal linear code if and  only if $e\equiv 3 \pmod 4$.
\end{theorem}
\begin{proof}
    If $e\equiv 3 \pmod 4$, then, by \Cref{lemma:weights}, the ratio between the minimum and maximum weights of $\cC^3_{\varepsilon}$ is
   \[
   \frac{| \cV^3_{\varepsilon}|-q^3-q^2-1}{| \cV^3_{\varepsilon}|-3q^2-1}=\frac{q^5-q^3+2q^2}{q^5}>\frac{q^2-1}{q^2};
\]
this proves the minimality by Condition~\eqref{suff}.
 If $e\equiv 1 \pmod 4$, then   $\langle \Pi_{\infty}\cap \cV^3_{\varepsilon} \rangle \neq \Pi_{\infty}$  and hence,
from \Cref{thm:min_code}, the code  $\cC^3_\varepsilon$ is not minimal. 
\end{proof}

\begin{corollary}\label{coro:dist}
  The code $\cC_{\varepsilon}^3=\cC(\cV^3_{\varepsilon})$ is a $[\vert \cV_\varepsilon^3 \vert, 4]$-linear code with minimum distance  $d=\vert \cV_\varepsilon^3 \vert- q^3-q^2-1$. Moreover, its $m$-th higher minimum
  weights $d_m$ with $m \in \{1,2,3\}$ are:
\[  d_1=|\cV^3_{\varepsilon}|-q^3-q^2-1;\;\; d_2=|\cV^3_{\varepsilon}|-q^2-1;\;\;  d_3=|\cV^3_{\varepsilon}|-1.
\]
\end{corollary}
\begin{remark}
We stress that in \cite{violating}, the authors provided a general method to transform a minimal code satisfying the Ashikhmin–Barg condition into a minimal code that violates it.
\end{remark}

\subsection{The variety $\cV_{\varepsilon}^r$, $r\geq 4$}

In this section we are going to  determine the parameters of  the linear code $\cC^r_{\varepsilon}=\cC(\cV_{\varepsilon}^r)$ generated by the projective points of $\cV^r_{\varepsilon}$ in $\PG(r,q^2)$, $r\geq 4$  showing that it is a minimal code for any odd integer $e$.

\begin{lemma}\label{min}
  Let $r \geq 2$ and $q=2^e$ with $e \geq 3$ odd.
  Then  the Fermat hypersurface $\cF_n^r$ of degree $n=2^{\frac{e-1}{2}}+1$ in $\PG(r,q^2)$ defined by the equation
  $$\cF^r_{n}: X_0^{2^{\frac{e-1}{2}}+1}+X_1^{2^{\frac{e-1}{2}}+1}+\ldots+X_{r}^{2^{\frac{e-1}{2}}+1}=0$$
  spans $\PG(r,q^2)$, that is
  $\langle \cF^r_{n} \rangle =\PG(r,q^2)$.
When $e=1$ the hypersurface $\cF_n^r$ is a hyperplane in $\PG(r,4)$.
\end{lemma}
\begin{proof}
  If $e=1$, then $\cF^r_n$ is a degenerate quadric in characteristic $2$,
  splitting in a hyperplane counted twice.
  Assume now $e\geq2$ and set $n=2^{\frac{e-1}{2}}+1$. 
  The proof is by induction on $r$. Suppose $r=2$.
  Since $\cF^2_{n}$ is non-singular and absolutely irreducible, the Hasse-Weil bound assures that the number $N_{q^2}:=N_{q^2}(\cF^r_{n})$ of its $\GF{q^2}$-rational points is at least
$q^2+1-(n-1)(n-2)q$.
Since each line of the projective plane meets $\cF^2_{n}$ in at most $n$ points and $N_{q^2}>n$ then the rational points on $\cF^2_{n}$ generate the plane.
Suppose now that the points of $\cF^{r-1}_{n}$ generate $\PG(r-1,q^2)$ and consider $\cF^{r}_{n} \cap \Pi$, where $\Pi$ is the hyperplane with equation $X_r=0$ of $\PG(r,q^2)$. Then $\cF^r_{n} \cap \Pi= \cF^{r-1}_{n}$; so, by induction, $\cF^{r-1}_{n}$  generates the hyperplane $\Pi$.
On the other hand, the point $M(1,0,\ldots,0,1)$ belongs to $\cF^r_n \setminus \Pi$ and hence
$$ \PG(r,q^2)= \langle M, \Pi \rangle = \langle M, \cF^{r-1}_{n} \rangle \subseteq \langle \cF_{n}^r \rangle.   $$
\end{proof}

\begin{theorem}\label{th:inter_hyp} For $r\geq4$ and $e\geq3$ odd, the set of rational points $\cV^r_{\varepsilon} \subset \PG(r,q^2)$ is cutting, i.e. for any hyperplane of $\PG(r,q^2)$,
\begin{equation}\label{span}
\langle \Pi \cap \cV^r_{\varepsilon} \rangle = \Pi.
\end{equation}
\end{theorem}
\begin{proof}
  Firstly, we prove that
  $ \langle \cV^r_{\varepsilon} \cap \Pi_{\infty}\rangle=\Pi_{\infty}$.
  Put $\cV^r_{\varepsilon, \infty}=\cV^r_{\varepsilon} \cap \Pi_{\infty}$.
  By Equation~\eqref{eq:V}, $\cV^r_{\varepsilon,\infty}$ has equations 
$$X_0=X_1^{2^{\frac{e-1}{2}}+1}+X_2^{2^{\frac{e-1}{2}}+1}+\ldots+X_{r-1}^{2^{\frac{e-1}{2}}+1}=0.$$
Then, this variety is a cone through the point $P_{\infty}(0,0,\ldots,1)$ and a Fermat hypersurface $\cF^{r-2}_{n}$, $n=2^{\frac{e-1}{2}}+1$, contained in the $(r-2)$-dimensional subspace $\Pi_{r-2}: X_0=X_r=0$.
By Lemma \ref{min}, this variety spans  $\Pi_{r-2}$ for $e\geq3$. So, since  $P_{\infty}$ belongs to   $\mathcal V^{r}_{\varepsilon,\infty}\setminus \Pi_{r-2}$, we have
\[ \langle \cV^r_{\varepsilon,\infty} \rangle =\langle P_{\infty},  \mathcal V^r_{\varepsilon, \infty} \cap \Pi_{r-2}\rangle = \langle P_{\infty}, \Pi_{r-2} \rangle =\Pi_{\infty}. \]
In order to prove \eqref{span}, for any other hyperplane $\Pi$ distinct from $\Pi_{\infty}$, it is enough to note that $$\vert \cV^r_{\varepsilon} \cap \AG(r,q^2) \cap \Pi \vert = \vert \cH^r_{\varepsilon} \cap \AG(r,q^2) \cap \Pi \vert,$$
and to show that  the size of the set $\cH^r_{\varepsilon} \cap \AG(r,q^2) \cap \Pi$ is strictly greater than $\vert \PG(r-2,q^2) \vert = \theta_{q^2}(r-2)$ points. 

We will use the following property, see \cite{A}.
Let $\cH^r:=\cH(r,q^2)$ be the Hermitian variety of equation $\cH^r: X_r^{q}+X_r=X_1^{q+1}+\ldots +X_{r-1}^{q}$ and $\Pi$ be a hyperplane through the point $P_{\infty}$.
Then 
\begin{equation}
\label{hPP}
\cH^r \cap \Pi \cap \Pi_{\infty} = \cH_{\varepsilon}^r \cap  \Pi  \cap \Pi_{\infty}.
\end{equation}
We split the discussion in four cases:

\begin{enumerate}[label=\roman*)]
\item  $P_{\infty} \not \in \Pi$ and $\vert \cH^r_{\varepsilon} \cap \Pi \vert =\vert \cH(r-1,q^2) \vert$ see \cite{A}.
In this case  $\cH^r_{\varepsilon} \cap \Pi_{\infty} \cap \Pi$ is  an $\cH(r-2,q^2)$, otherwise $P_\infty$ would belong to $\Pi$ against the hypothesis, cf. \eqref{eq:Hinf}. Hence,
\begin{equation}\label{eq:span}
\begin{split}
 \vert \cH^r_{\varepsilon} \cap \AG(r,q^2) \cap \Pi \vert&  = \vert \cH(r-1, q^2) \vert -\vert  \cH(r-2, q^2) \vert \\
 &=q^{r-2}(q^{r-1}-(-1)^{r-1}).
 \end{split}
 \end{equation}

Let us consider
\begin{equation}\label{case1}
\begin{split}
(q^2-1) & \left (\vert \cH^r_{\varepsilon} \cap \AG(r,q^2) \cap \Pi \vert - \theta_{q^2}(r-2) \right)\\
&=(q^2-1)q^{r-2}(q^{r-1}-(-1)^{r-1})-(q^{2(r-1)}-1)\\
&=\begin{cases}
  (q^{r-1}+1)(q^r-q^{r-2}-q^{r-1}+1)  &\, r \textnormal{ even }\\
   (q^{r-1}-1)(q^r-q^{r-2}-q^{r-1}-1) &\, r \textnormal{ odd }.
\end{cases}
\end{split}
\end{equation}
In order to get the claim, it is enough to show that $q^{r}-q^{r-2}-q^{r-1}-1$ is positive for any $q \geq 2$ and $r \geq 4$. \\
Since $q^{r-\ell} \leq \frac{q^{r}}{2^\ell}$ for any $q \geq 2$, $r \geq 4$ and $\ell \in \{0, \ldots, r\}$ and by \eqref{case1}, we have
\[
q^r-q^{r-2}-q^{r-1}-1 \geq \frac{q^r}{4}-1> 0.
\]

\item $P_{\infty} \not \in \Pi$ and $\vert \Pi \cap \cH^r_{\varepsilon} \vert =\vert  \Pi_0 \cH(r-2,q^2) \vert$. Since $P_{\infty}\notin \Pi$, we have that $\cH^r_{\varepsilon}\cap \Pi_{\infty}\cap \Pi$ is again an $\cH(r-2,q^2)$. We proceed as in point i):
\begin{equation}
\begin{split}
\vert \cH^r_{\varepsilon} \cap \AG(r,q^2) \cap \Pi \vert&=
|\Pi_0\cH(r-2,q^2)|-|\cH(r-2,q^2)|=\\
&=(-1)^{r-1}q^{r-1}+\frac{(q^{r-1}-(-1)^{r-1})(q^r-q^{r-2})}{(q^2-1)}\\
&=(-1)^{r-1}q^{r-1}+q^{r-2}(q^{r-1}-(-1)^{r-1}).
\end{split}
\end{equation}
Studying the case for $r$ even, we obtain
\[
\vert \cH^r_{\varepsilon} \cap \AG(r,q^2) \cap \Pi \vert=
q^{2r-3}-q^{r-1}+q^{r-2}=q^{r-2}(q^{r-1}-q+1)
\]
and hence,
\begin{multline*}
(q^2-1)\left ( \vert \cH^r_{\varepsilon} \cap \AG(r,q^2) \cap \Pi \vert - \theta_{q^2}(r-2) \right )=\\
(q^2-1)q^{r-2}(q^{r-1}-q+1)-(q^{2(r-1)}-1).
\end{multline*}
Then, we have that 
\begin{equation}
\begin{split}
q^{2r-1}&+q^r-q^{r+1}-q^{2r-3}-q^{r-2}+q^{r-1}-q^{2(r-1)}+1\\
&>q^{r-2}(q^{r+1}+q^2-q^3-q^{r-1}-1+q-q^r)\\
&>q^{r-2}(q^{r+1}+q^2-q^3-q^{r-1}-q^r)\\
&>q^{r-2}(q^{r-1}(q^2-q-1)+q^2-q^3)\\
&\geq q^{r}(q^{r-3}+1-q)>0,
\end{split}
\end{equation}
for any $r \geq 4$ and $q \geq 2$. Then,
\[
\vert \cH^r_{\varepsilon} \cap \AG(r,q^2) \cap \Pi \vert > \theta_{q^2}(r-2).
\]
Similarly in the case of $r$ odd we have that
\[
\vert \cH^r_{\varepsilon} \cap \AG(r,q^2) \cap \Pi \vert=
q^{r-1}+q^{2r-3}-q^{r-2}
\]
and therefore 
\begin{equation}
\begin{split}
(q^2-1)&\left(\vert \cH^r_{\varepsilon} \cap \AG(r,q^2) \cap \Pi \vert - \theta_{q^2}(r-2) \right)\\
&>q^{r-2}(q^{r+1}-q^r-q^{r-1}+q^{3}-q^{2}-q+1)\\
&>q^{r-2}(q^{r-1}(q^2-q-1)+q(q^2-q-1)+1)\\
&>q^{r-1}(q^{r-2}+1)(q^2-q-1)>0
\end{split}
\end{equation}
for any $q\geq 2$.

\item $P_{\infty}  \in \Pi$ and $\vert \Pi \cap \cH^r_{\varepsilon} \vert =\vert \Pi_0 \cH(r-2,q^2) \vert$.
%Observe that $\Pi\cap\Pi_{\infty}\cap\cH_{\varepsilon}^r$ is
%either a $\Pi_0\cH(r-3,q^2)$ or a $\Pi_1\cH(r-%4,q^2)$.

 We know that $|\cH^r\cap\Pi|=|\cH_{\varepsilon}^r\cap\Pi|$ and in particular the hyperplane $\Pi$ is tangent to $\cH^r$ at some point $Q$ different
from $P_{\infty}$. By~\eqref{hPP}, $\cH^r\cap\Pi\cap\Pi_{\infty}=\cH_{\varepsilon}^r\cap\Pi\cap\Pi_{\infty}$ hence, it follows that 
$\cH_{\varepsilon}^r\cap\Pi\cap\Pi_{\infty}=\cH^r\cap P_{\infty}^{\perp}\cap Q^{\perp}$ with $P_{\infty}\in Q^{\perp}$; so
$|\cH_{\varepsilon}^r\cap\Pi_{\infty}\cap\Pi|=|\Pi_1\cH(r-4,q^2)|$.

%One can easily check that if $r$ is even %$|\Pi_0\cH(r-3,q^2)|>|\Pi_1\cH(r-%4,q^2)|$ while if $r$ is odd $|\Pi_0\cH(r-%3,q^2)|<|\Pi_1\cH(r-4,q^2)|$.

%Therefore, in order to prove that $\langle \cV^r_{\varepsilon}\cap \Pi \rangle=\Pi$ in the case of $r$ even, it is enough to show
%\[
%\vert \Pi_0\cH(r-2,q^2)\vert -\vert %\Pi_0\cH(r-3,q^2) \vert >\theta_{q^2}(r-2).
%\]
%Making explicit the calculation, we have
%\begin{equation}
%\begin{split}
%&\vert \Pi_0\cH(r-2,q^2)\vert -\vert %\Pi_0\cH(r-3,q^2) \vert\\
%&=\frac{(q^{r}-1)(q^{r-1}+1)}{q^2-1}
%-q^{r-1}-\frac{(q^{r-1}+1)(q^{r-2}-1)}{q^2-1}-q^{r-2}\\
%&=q^{2r-3}-q^{r-1}.
%\end{split}
%\end{equation}
%Then,
%\begin{multline*}
%(q^2-1)(\vert\cH^r_{\varepsilon}\cap %\AG(r,q^2)\cap \Pi \vert-\theta_{q^2}(r-2))\geq\\
%(q^2-1)(q^{2r-3}-q^{r-1})-q^{2(r-1)}+1,
%\end{multline*}
%and 
%\begin{equation}
%\begin{split}
%q^{2r-1}-&q^{r+1}-q^{2r-3}+q^{r-1}-q^{2(r-1)}+1\\
%&=q^{2r-3}(q^{2}-q-1)-q^{r+1}+q^{r-1}+1\\
%&\geq q^{2r-3}-q^{r+1}+q^{r-1}+1\\
%&>q^{2r-3}-q^{r+1}=q^{r+1}(q^{r-4}-1) \geq 0
%\end{split}
%\end{equation}
%since  $r>3$ and $q^2-q-1\geq 1$ for $q\geq 2$.

Therefore, in order to prove that $\langle \cV^r_{\varepsilon}\cap \Pi \rangle=\Pi$, it is enough to show
\[
\vert \Pi_0\cH(r-2,q^2)\vert -\vert \Pi_1\cH(r-4,q^2) \vert >\theta_{q^2}(r-2).
\]
In case of $r$ odd we have 
%\[
 %\vert \Pi_0\cH(r-2,q^2)\vert -\vert \Pi_1\cH(r-4,q^2) \vert >\theta_{q^2}(r-2).
%\]

We observe that 
\begin{multline*}
\vert \Pi_0\cH(r-2,q^2)\vert -\vert \Pi_1\cH(r-4,q^2) \vert=\\
\frac{(q^r+(-1)^{r-1})(q^{r-1}-(-1)^{r-1}}{q^2-1}+
\\
(-1)^{r-1}q^{r-1}-\frac{(q^{r-3}-(-1)^{r-4})(q^{r-4}-(-1)^{r-4})q^4}{q^2-1}-q^2-1
\end{multline*}
which is equal to $q^{2r-3}$ for $r$ even as well as for $r$ odd.

Now,
\begin{multline*}
(q^2-1)(\vert \Pi_0\cH(r-2,q^2)\vert -\vert \Pi_1\cH(r-4,q^2) \vert- \theta_{q^2}(r-2))\\
=q^{2r-1}-q^{2r-3}-q^{2r-2}+1
=q^{2r-3}(q^2-q-1)+1
\end{multline*}
where the last term is clearly positive for $q\geq 2$.

\item $P_{\infty}  \in \Pi$, $\vert \Pi \cap \cH^r_{\varepsilon} \vert =\vert \cH(r-1,q^2) \vert$.

In this case $ \vert\Pi \cap \Pi_{\infty} \cap \cH^r_{\varepsilon}\vert  =\vert \Pi_0\cH(r-3,q^2)\vert$.
Then,
\begin{multline}
\vert\cH^r_{\varepsilon}\cap \AG(r,q^2)\cap \Pi \vert=\\
\vert \cH(r-1,q^2)\vert -\vert \Pi_0\cH(r-3,q^2) \vert=\frac{(q^{r}+(-1)^{r-1})(q^{r-1}-(-1)^{r-1})}{q^2-1}\\
-\frac{(q^{r-1}+(-1)^{r-2})(q^{r-2}-(-1)^{r-2})}{q^2-1}-(-1)^{r-2}q^{r-2}
\end{multline}
Both for $r$ even and odd, we get
\[\vert\cH^r_{\varepsilon}\cap \AG(r,q^2)\cap \Pi \vert=
\vert \cH(r-1,q^2)\vert -\vert \Pi_0\cH(r-3,q^2) \vert=q^{2r-3}
.\]
 Thus, we can argue as in the case above.
%\begin{multline}
%(q^2-1)(\vert\cH^r_{\varepsilon}\cap %\AG(r,q^2)\cap S \vert-\theta_{q^2}(r-2) )\\
%=(q^2-1)q^{2r-3}-q^{2(r-1)}+1\\
%>(q^2-q-1)q^{2r-3}+1,
%\end{multline}
%and this is strictly greater than zero.
\end{enumerate}
\end{proof}

\begin{theorem} \label{th:cone-over-Fermat}
    Let $\cV^r_{\varepsilon} \subset \PG(r,q^2)$, $ r \geq 4 $. Then,
    $$\vert \cV^r_{\varepsilon} \vert = q^{2r-1}+\vert \cF^{r-2}_{n}\vert q^2 +1,$$
where $n=2^{\frac{e-1}{2}}+1$.
\end{theorem}
\begin{proof}
    The result follows observing that $\cV^r_{\varepsilon} \cap \Pi_{\infty}$ is a cone with vertex the point $P_{\infty}(0,0,\ldots,1)$ over a Fermat hypersurface of the $(r-2)$-dimensional subspace with equations $X_0=X_r=0$. 
\end{proof}

By \Cref{th:cone-over-Fermat}, in order to get the length and the weights of the projective linear code $\cC_{\varepsilon}^r$, $r \geq 4 $, it is necessary to determine  the size of the Fermat hypersurface $\cF_{n}^{r-2}$, $r \geq 4$ and $n= 2^\frac{e-1}{2}+1$, lying in the $(r-2)$-dimensional subspace of $\PG(r,q^2)$ of equations $X_0=X_r=0$. Thus, we prove the following.  
\begin{proposition}\label{thm:cardinalityinf}
The number $N_{q^2}$ of $\GF{q^2}$-rational points of $\cF^{r}_{n}$ in $\PG(r,q^2)$, $r\geq 2$,  satisfies the following properties:
\begin{enumerate}[label=\upshape{(\roman*)}]
\item if $e\equiv 1 \pmod 4$ then $N_{q^2}=\theta_{q^2}(r-1)$;
\item if $e\equiv 3 \pmod 4$  then $$N_{q^2}\leq (n-1)q^{2(r-1)}+nq^{2(r-2)}+\theta_{q^{2}}(r-3).$$
\item if $e=3$ and $r=2$ then  $N_{q^2}=(q+1)^2.$
\end{enumerate}
\end{proposition}
\begin{proof}
\begin{enumerate}[label=(\roman*)]
\item
  If $e=1$, then $\cF_{n}^r$ is a hyperplane in $\PG(r,q^2)$ and the
  statement holds. Suppose now $e\geq3$.
The proof is by induction on $r$. For $r=2$,    $\mathcal F^2_{n}$ is the  plane curve of equation $X_0^n+X_1^n+X_2^n=0$.
Then, the points of  $\cF^2_{n}$ 
are of the following forms:
\begin{enumerate}[label=(\alph*)]
\item\label{xA} $(0, 1, x_2)$, $x_2^n = 1$;
\item\label{xB} $ (1, x_1, 0)$, $ x_1^n = 1$;
\item\label{xC} $(1, x_1, x_2)$, $ x_1^n = \alpha \neq 1$ and $x_2^n=1+\alpha\neq 0$, where $\alpha \in \GF{q^2}$.
\end{enumerate}
 Since $\gcd(q^2-1,n)=\gcd(2^{2e}-1,2^{\frac{e-1}{2}}+1)=1$ for  $e \equiv 1 \pmod 4$,  then any $m\in\GF{q^2}\setminus\{0\}$ is an $n$-th power of an element in $\GF{q^2}$. Furthermore, there is
precisely  one point of type \ref{xA} and one point of type \ref{xB}.
As for type \ref{xC}, $x_1^n=1$ implies $x_1=1$; so
there are exactly $q^2-1$ possible second coordinates
for any point of type \ref{xC}. Given any such second coordinate, the third coordinate $x_2$ is uniquely determined. So, there are precisely
$q^2-1$ points of type \ref{xC}. So,  $\cF^2_{n}$ has $q^2+1$ rational points.
% (it is indeed a rational curve).
Now, by the induction hypothesis, let suppose that the points of $\cF^{r-1}_{n}$ are 
$\theta_{q^2}(r-2)=q^{2(r-2)}+q^{2(r-3)}+\ldots+q^2+1$ in total and consider the hypersurface   $\cF^{r}_{n}:\ X_0^n+X_2^n+\ldots+X_{r}^n=0$ of $\PG(r,q^2)$.
The points of $\cF^{r}_{n}$  for which $x_{r}=0$ are  $\theta_{q^2}(r-2)$ by induction, whereas the points of our hypersurface for which $x_{r}=1$ are $q^{2(r-1)}$ in total; so the result follows.
\item This is a consequence of \Cref{RationalVar1} applied to the hypersurface $\cF^{r}_{n}$.
\item  In this case, the plane  curve $\cF^2_{3}$ turns out to be maximal by Theorem \cite[Theorem 10.65]{HKT}. So the result follows from the
  Hasse-Weil bound.
\end{enumerate}
\end{proof}

\begin{theorem}\label{th:code4}
The linear code $\cC^r_{\varepsilon}$ generated by the projective points of $\cV^r_{\varepsilon}$ in $\PG(r,q^2)$, $r\geq 4$ is a $(r+1)$-dimensional minimal code for any odd integer $e$ with length $\vert \cV_{\varepsilon}^r \vert$.
\end{theorem}
\begin{proof} 
The thesis follows from \Cref{thm:min_code} and \Cref{th:inter_hyp}.
\end{proof}

Now, consider the multiset
 $\overline{\cV_{\varepsilon}^r}$ in $\PG(r,q^2)$, $r\geq 3$, that consists of the points in $(\cV_{\varepsilon}^r \setminus \Pi_{\infty})$ 
  plus the point $P_{\infty}$, which is assigned multiplicity  $j\geq 1$. Using an approach similar to that used in \cite{AK} we  get the following result.
 \begin{theorem}\label{thm:codeVbar}
    The linear code associated with  $\overline{\cV_{\varepsilon}^r}$ is a $[q^{2r-1}+j, r+1]_{q^2}$ code with
 weights: \[q^{2r-1}, q^{2r-1}-q^{2r-3},q^{2r-1}-q^{2r-3}+(-1)^{r-1}q^{r-2}+j,\]
 \[ q^{2r-1}-q^{2r-3}+(-1)^{r-1}q^{r-2}-(-1)^{r-1}q^{r-1}+j.\] Furthermore, it  is a   $3$-weight code in the following cases
 \begin{enumerate}[label=(\roman*)]
\item  $ r$ odd and  $j = q^{r-1}-q^{r-2}$ or $j=q^{2r-3}-q^{r-2}$ or $j=q^{2r-3}+q^{r-1}-q^{r-2}$;
\item  $r$ even  and $j=q^{r-2}$, or $j=q^{2r-3}+q^{r-2}$ or $j=q^{2r-3}-q^{r-1}+q^{r-2}$. 
\end{enumerate}
 \end{theorem}
 \begin{proof}
It follows from  the proofs of \Cref{lemma:weights}, \Cref{lemma:weights2} and
 \Cref{th:inter_hyp}.
\end{proof}

%\footnote{For $r$ odd, if $j=q^{2r-3}+q^{r-1}-q^{r-2}$ the weights become 
%$$\{q^{2r-1},q^{2r-1}-q^{2r-3},q^{2r-1}+q^{r-1}\}$$}
\section{The cutting gap}
\label{cutgap}

Minimal projective linear codes and cutting blocking sets with respect to the hyperplanes are essentially the same kind of object, see~\cite{BB} and \cite{TQLZ}. In~\cite{DGMP},  the notion of $t$-fold strong blocking set $\Lambda$ is introduced, as a set of points in a projective space $\PG(n,q)$ such that every $(t-1)$-dimensional projective subspace is spanned by $t$ points in $\Lambda$. Clearly, being a cutting blocking set with respect to the hyperplanes and being a $n$-fold strong blocking set mean the same.

In this section we introduce
the concept of \emph{cutting gap}, as a measure of the amount in which
a set $\Lambda$ fails to be a $t$-fold strong blocking set for some $t$.
It provides a quantitative measure of how effectively a point set intersects $k$-dimensional subspaces in a projective space. This notion is closely related to the geometry of higher Hamming weights and has applications in coding theory and incidence geometry, as first explored in~\cite{DGMP}. In particular, cutting gaps allow us to distinguish between point sets that merely intersect subspaces and those that span them, offering a refined tool for analyzing the structure of algebraic varieties.

Throughout this section, we use projective dimensions and adopt the convention $\dim(\emptyset) = -1$.

\begin{definition}
Let $\Omega$ be a non-empty set of points in $\PG(r,\KK)$, and let $0 \leq k \leq r$. The \emph{$k$-th cutting gap} of $\Omega$ is defined as
\[
\tau_k(\Omega) = k - \min\left\{ \dim\left(\langle \Pi \cap \Omega \rangle \right) \colon \dim(\Pi) = k \right\}.
\]
\end{definition}
According to~\cite{DGMP}, $\Omega$ is a \emph{$(k+1)$-fold strong blocking set} if and only if $\tau_k(\Omega) = 0$. In this paper, we refer to this property as being \emph{$k$-cutting}, and we say that $\Omega$ is a \emph{$k$-cutting set} if $\tau_k(\Omega) = 0$. For explicit constructions of $k$-cutting sets for selected values of $k$, see~\cite[\S 3]{DGMP}.

Note that $\Omega$ is a classical \emph{cutting set with respect to hyperplanes} precisely when $\tau_{r-1}(\Omega) = 0$. More generally, $\Omega$ is $k$-cutting if every $k$-dimensional subspace is spanned by its intersection with $\Omega$.

\begin{theorem}
Let $\Omega$ be a $k$-cutting set in $\PG(r,\KK)$. Then $\Omega$ is also $\ell$-cutting for all $k \leq \ell \leq r$.
\end{theorem}

\begin{proof}
If $k = r$, the statement is trivial. Otherwise, suppose $\tau_k(\Omega) = 0$. Any $(k+1)$-dimensional subspace $\Pi$ can be expressed as the union of its $k$-dimensional subspaces:
\[
\Pi = \bigcup_{\substack{\Sigma \leq \Pi \\ \dim(\Sigma) = k}} \Sigma = \bigcup_{\substack{\Sigma \leq \Pi \\ \dim(\Sigma) = k}} \langle \Sigma \cap \Omega \rangle \subseteq \langle \Pi \cap \Omega \rangle = \Pi.
\]
Hence, $\tau_{k+1}(\Omega) = 0$, and the result follows by induction.
\end{proof}

We now compute the cutting gaps for non-degenerate Hermitian varieties using properties of the induced polarity.

\begin{theorem}
\label{cut-H}
Let $\cH^r = \cH(r,q^2)$ be a non-degenerate Hermitian variety in $\PG(r,q^2)$. Then,
\[
\tau_{r-t}(\cH^r) =
\begin{cases}
0 & \text{if } t > \left\lfloor \frac{r}{2} \right\rfloor, \\
1 & \text{if } t \leq \left\lfloor \frac{r}{2} \right\rfloor.
\end{cases}
\]
\end{theorem}

\begin{proof}
Hermitian varieties over finite fields always contain at least one rational point. Thus, any subspace $\Pi = \PG(V_{k+1})$ of dimension $k$ in $\PG(r,q^2)$ intersects $\cH^r$ in a blocking set for the lines of $\Pi$, which either spans $\Pi$ or is a hyperplane of it.

Let $H_{r+1}$ be the non-singular Hermitian matrix defining $\cH^r$. For any subspace $\Pi$ of codimension $t$, the induced Hermitian form has rank at least $r+1 - 2t$, as it corresponds to a $(r+1-t) \times (r+1-t)$ minor of $H_{r+1}$. If $r+1 - 2t > 1$, i.e., $t < r/2$, the form has rank at least $2$, and $\Pi \cap \cH^r$ spans $\Pi$, so $\tau_{r-t}(\cH^r) = 0$.

Conversely, if $t \geq r/2$, there exist subspaces $\Pi$ such that $\Pi \cap \cH^r$ is a hyperplane of $\Pi$, implying $\tau_{r-t}(\cH^r) = 1$.
\end{proof}

We leave as an open problem whether the list of cutting gaps in Theorem~\ref{cut-H} characterizes Hermitian varieties among quasi-Hermitian ones.

In $\PG(2,q)$, the only relevant value is $k = 1$. Then:
 $\tau_1(\Omega) = 0$ if and only if $\Omega$ is a $2$-fold blocking set for the lines;
 $\tau_1(\Omega) = 1$ if $\Omega$ is a blocking set but not $2$-fold;
 $\tau_1(\Omega) = 2$ if there exist lines disjoint from $\Omega$.
This provides lower bounds on the size of $\Omega$.

A trivial observation is that $\tau_k(\Omega) > k$ if and only if there exist $k$-subspaces disjoint from $\Omega$. In such cases, $\tau_k(\Omega)$ does not offer any further geometric insight. To refine this, we introduce the notion of \emph{modified cutting gaps}.

Let $-1 \leq s < k$. Define:
\[
\tau_{k,s}(\Omega) = k - \min\left\{ \dim\left(\langle \Pi \cap \Omega \rangle \right) \colon \dim(\Pi) = k, \dim(\langle \Pi \cap \Omega \rangle) \geq s \right\}.
\]
By definition, $\tau_{k,-1}(\Omega) = \tau_k(\Omega)$. Moreover, $\tau_{k,s}(\Omega) \leq \tau_{k,s'}(\Omega)$ whenever $s > s'$, so the sequence is non-increasing in $s$ for fixed $k$. If $\Omega$ spans at least an $s$-space, then $\tau_{k,s}(\Omega) \leq k - s$. While equality is expected in general, the following example shows that this is not always the case.

\begin{example}
Let $\Omega := \mathcal{Q}^{-}(3,q) \setminus \{P\}$ be the set of points of an elliptic quadric in $\PG(3,q)$ with one point $P$ removed. Then:
 $\tau_2(\Omega) = \tau_{2,-1}(\Omega) = 3$, since the tangent plane at $P$ is disjoint from $\Omega$;
 $\tau_{2,0}(\Omega) = 2$, as tangent planes at other points intersect $\Omega$ in a single point;
$\tau_{2,1}(\Omega) = 0$, since any plane intersecting $\Omega$ in at least two points meets it in a conic or a conic minus $P$.
\end{example}

We leave to future work a deeper investigation of the properties of modified cutting gaps and the classification of point sets for which $\tau_{k,s}(\Omega) < k - s$ for some values of $k$ and $s$.

\subsection{Intersection of $\cV^4_{\varepsilon}$ with lines} \label{subsec:ext_lines}
We now study the intersections of $\cV^4_{\varepsilon}$ with lines, as
to show that $\tau_1(\cV^4_{\varepsilon})=2$, i.e. $\cV^4_{\varepsilon}$ is
not a blocking set for the lines of $\PG(4,q^2)$.

Let $\Pi_0$ be the plane of equation $X_0=X_4=0$.
As seen before, $\Pi_0\cap \cV^4_{\varepsilon}$ is the Fermat curve of $\Pi_0$ with equation:
 \begin{equation}
\cF^2_{s+1}:X_1^{s+1}+X_2^{s+1}+X_3^{s+1}=0.
\end{equation}
where $s=2^\frac{e-1}{2}$. In the following, we are interested
in determining an external line to $\cV^4_{\varepsilon}$; so, it is enough to
find a line in $\Pi_0$ that does not intersect $\cF^2_{s+1}$.
We split the discussion depending on the value of integer $e$ modulo $4$.

\subsubsection{Case $e\equiv 1 \pmod 4$, $e\neq1$}
Let $\ell$ be  a line in the plane $\Pi_0$ with equations of the form $X_0=\alpha X_1+\beta X_2+X_3=X_4=0$. Our aim is to prove that there is at least a pair $(\alpha,\beta)\in \GF{q^2}^2\setminus \{(0,0)\}$ such that $\ell\cap \cF^2_{s+1}=\emptyset$ or, equivalently that the following system in $X_1,X_2,X_3$ has no non-trivial solutions:
\begin{equation} \label{intersection-e=1}
    \begin{cases}
    X_1^{s+1}+X_2^{s+1}+X_3^{s+1}=0\\
    \alpha X_1+\beta X_2+X_3=0.
    \end{cases}
\end{equation}
Substituting the second equation of System~\eqref{intersection-e=1} in the first, we obtain
\[
(1+\alpha^{s+1})X_1^{s+1}+\alpha^s \beta X_1^{s}X_2+\alpha\beta^{s}X_1X_2^s+(1+\beta^{s+1})X_{2}^{s+1}=0.
\]

\begin{proposition}\label{Nic}
 Let $e \equiv 1 \pmod 4$, $e > 1$ and $s= 2^{\frac{e-1}{2}}$. Then for every $\alpha \notin \GF{4}$ there exists $1\neq\beta \in \GF{q^2}$ such that the equation
\begin{equation}\label{eq:tuv}
(1+\alpha^{s+1})X_1^{s+1}+\alpha ^s \beta X_1^{s}X_2+\alpha\beta^{s}X_1X_2^s+(1+\beta^{s+1})X_{2}^{s+1}=0.
\end{equation}
has no non-trivial solutions in $\GF{q^2}^2$.
\end{proposition}
\begin{proof}
Since $X_2=0$ gives the trivial solution, we can assume  $X_2 \neq 0$. Moreover, since $\alpha \neq 1$, substituting $X= X_1/X_2$ in \eqref{eq:tuv}, we get
\[
X^{s+1}+\frac{\alpha^{s}\beta}{1+\alpha^{s+1}}X^{s}+\frac{\alpha\beta^{s}}{1+\alpha^{s+1}}X+\frac{1+\beta^{s+1}}{1+\alpha^{s+1}}=0.
\]
Putting
\begin{equation}
t = \frac{\alpha^s \beta}{1 + \alpha^{s+1}}, \qquad u = \frac{\alpha \beta^s}{1 + \alpha^{s+1}}, \,\, \textnormal{ and } \,\, v = \frac{1 + \beta^{s+1}}{1 + \alpha^{s+1}},
\end{equation}
the equation
\begin{equation}\label{eq:t-u-v}
    X^{s+1} + tX^s + uX + v = 0
\end{equation}
has no solution in $\GF{q^2}$ if and only if Equation \eqref{eq:tuv} has no non-trivial solution in $\GF{q^2}^2$.
 Consider the polynomial $f_a=X^{s+1}+X+a$ with $a \in \GF{q^2}$. By \cite[Theorem 5.6]{B},  the set  $N_0=\{a \in \GF{q^2} \colon f_a \textnormal{ has no roots in }\GF{q^2}\}$ has size $\frac{2}{5}(2^{2e}+1) \geq 410$ for $e \geq 5$. Then, given $\alpha\in \GF{q^2} \setminus \GF{4}$, we may choose $a\in N_0$ such that 
\begin{equation} \label{conditions:a}
a\neq \frac{(\alpha^{s+1}+1)^{\frac{(s+1)^2}{s}-2}}{(\alpha+\alpha^{s^2})^{\frac{s+1}{s}}} \,\, \textnormal{ and }\,\, a\neq \frac{\alpha^{s+1}(\alpha^{s+1}+1)^{\frac{(s+1)^2}{s}-2}}{(\alpha+\alpha^{s^2})^{\frac{s+1}{s}}}.
\end{equation}
Take now $\beta \in \GF{q^2}$ such that
\begin{equation}\label{n}
\beta^{s+1}=\frac{(1+\alpha^{s+1})^{\frac{(s+1)^2}{s}-1}}{a(\alpha+\alpha^{s^2})^{\frac{s+1}{s}}+(1+\alpha^{s+1})^{\frac{(s+1)^2}{s}-2}}.
\end{equation}
Since every element of $\GF{q^2}$ is always an $(s+1)$-th power, such an element $\beta$ exists.
Firstly, it is straightforward to see that $\beta$ cannot be equal to $1$, otherwise $a$ is equal to the second expression in \eqref{conditions:a}, getting a contradiction.
Then, we will show that for this choice of $\alpha \in \GF{q^2} \setminus \GF{4}$ and $\beta$ as in \eqref{n}, Equation \eqref{eq:t-u-v} can be transformed in $X^{s+1}+X+a=0$ and hence it does not have any solutions in $\GF{q^2}$. 
Then, we prove that $u \neq t^ s$ and $v \neq u t$. Indeed, if $u=t^s$, we get 
\begin{equation}
    \frac{\alpha\beta^{s}}{1+\alpha^{s+1}}=\frac{\alpha^{s^2}\beta^s}{1+\alpha^{s^2+s}}.
\end{equation}
    This is equivalent to $\alpha^{s^2}=\alpha$ and hence $\alpha^{2^{e-1}-1}=1$. Since $\gcd(e-1,2e)=2$, $\alpha \in \GF{4}$ against the hypotheses. On the other hand, if $v=ut$, we have
   \begin{equation}
\frac{\beta^{s+1}+1}{\alpha^{s+1}+1}=\frac{\alpha^{s+1}\beta^{s+1}}{(\alpha^{s+1}+1)^2}
   \end{equation} 
and hence $\alpha^{s+1}+\beta^{s+1}+1=0$. Substituting $\beta^{s+1}=\alpha^{s+1}+1$
 in \eqref{n}, we get
 \begin{equation*}
     a(\alpha+\alpha^{s^2})^{\frac{(s+1)^2}{s}}(\alpha^{s+1}+1)=0.
 \end{equation*}
Since $\alpha \not \in \GF{4}$, this implies that $a=0$, a contradiction because $a \in N_0$.
By performing in \eqref{eq:tuv} the substitution $X \longmapsto (u+t^s)^{s^{-1}}X+t$, we get
 \begin{equation}\label{eq:tuvx}
 X^{s+1}+X+\frac{tu+v}{(u+t^{s})^{\frac{s+1}{s}}}=0.
 \end{equation}
In order to complete the proof, it is enough to prove that $\frac{tu+v}{(u+t^s)^{\frac{s+1}{s}}}=a$. Then,
\begin{equation*}
\frac{tu+v}{(u+t^s)^{\frac{s+1}{s}}}=\frac {\alpha^{s+1}+\beta^{s+1}+1}{(\alpha^{s+1}+1)^2}\cdot \frac{(1+\alpha^{s+1})^{\frac{(s+1)^2}{s}}}{\beta^{s+1}(\alpha+\alpha^{s^2})^{\frac{s+1}{s}}}.
\end{equation*}
Substituting \eqref{n} in the expression above, we get
\begin{equation*}
\begin{split}
&\frac{(\alpha^{s+1}+1)\left (a(\alpha+\alpha^{s^2})^{\frac{s+1}{s}}+(1+\alpha^{s+1})^{\frac{(s+1)^2}{s}-2}\right)+(1+\alpha^{s+1})^{\frac{(s+1)^2}{s}-1}}{(\alpha^{s+1}+1)^2 \left (a(\alpha+\alpha^{s^2})^{\frac{s+1}{s}}+(1+\alpha^{s+1})^{\frac{(s+1)^2}{s}-2} \right )} \cdot \\
& \frac{\left (a(\alpha+\alpha^{s^2})^{\frac{s+1}{s}}+(1+\alpha^{s+1})^{\frac{(s+1)^2}{s}-2})\right)(1+\alpha^{s+1})^{\frac{(s+1)^2}{s}}}{(\alpha+\alpha^{s^2})^{\frac{s+1}{s}}(1+\alpha^{s+1})^{\frac{(s+1)^2}{s}-1}}\\
&=\frac{(\alpha^{s+1}+1)^2\left (a(\alpha+\alpha^{s^2})^{\frac{s+1}{s}}+(1+\alpha^{s+1})^{\frac{(s+1)^2}{s}-2}\right)+(1+\alpha^{s+1})^{\frac{(s+1)^2}{s}}}{(\alpha+\alpha^{s^2})^{\frac{s+1}{s}}(1+\alpha^{s+1})^2}=a.
\end{split}
\end{equation*}
\end{proof}

\begin{comment}
\begin{remark}    
   Every line $\ell_m : X_0=mX_1+X_2=X_4=0$ meets the Fermat curve of equation $\cF_{s+1,2}$
where $s=2^\frac{e-1}{2}$.
Indeed, the points in these intersection satisfy the system
\begin{equation}\label{eq:systmx}
    \begin{cases}
        mX_1+X_2=0\\
        X_1^{s+1}+X_2^{s+1}+X_3^{s+1}=0.
    \end{cases}
\end{equation}
and hence the equation
  $$(m^{s+1}+1)X_1^{s+1}+X_3^{s+1}=0.$$
If $m=1$ and, hence $m^{s+1}=1$, then $\ell_m \cap \cF_{s+1,2}=P(0,1,1,0,0)$. If $m \neq 1$, then 
$\ell_m$ is external to $\cF_{s+1,2}$ if and only if 
$$\xi^{s+1}={m^{s+1}+1}$$
has no solution (here $\xi=X_3/X_1$). Since $\gcd(q^2-1,s+1)=\gcd(2^{2e}-1,2^\frac{e-1}{2}+1)=1$ for $e \equiv 1 \pmod 4$,  any $ 1 \neq m \in \GF{q^2}$, ${m^{s+1}+1}$ is an $(s+1)$-th power of an element in $\GF{q^2}$. In particular, the equation $\xi^{s+1}={m^{s+1}+1}$ has a unique solution for each $m \in GF(q^2)$  and hence we get that the number of rational points of $\cF_{s+1,2}$ is $q^2+1$.
\end{remark}
\end{comment}

\subsubsection{Case $e \equiv 3 \pmod 4$}
Suppose now $e \equiv 3 \pmod 4$ and recall that $\gcd(q^2-1,s+1)=3$. {We show that in this case there is at least one line of equation 
$\ell_{\alpha} : X_0=\alpha X_1+X_2=X_4=0$ that does not meet the Fermat curve $\cF^2_{s+1}$}. Indeed, the points in $\ell_{\alpha} \cap \cF^2_{s+1}$   satisfy the system
	 \begin{equation}\label{eq:systmx}
	 	\begin{cases}
	 		X_0=0\\
	 		\alpha X_1+X_2=0\\
	 		X_4=0\\
	 		X_1^{s+1}+X_2^{s+1}+X_3^{s+1}=0.
	 	\end{cases}
	 \end{equation}
Hence we need to   study the solutions of equation
 \[   (\alpha^{s+1}+1)X_1^{s+1}+X_3^{s+1}=0. \]
If $\alpha^{s+1}=1$, then $\alpha \in \{c_0=1,c_1, c_2\}$ with $c_i^{s+1}=1$, $i=0,1,2$. By the equation above, $X_3=0$  and the line $\ell_{c_i}$ meets the curve only at the point 
$(0,1,c_i,0,0)$. If $\alpha^{s+1}\neq 1$,  we can suppose $X_3\neq 0$ and
consequently also $X_1\neq 0$. Putting $X=X_3/X_1$, we study the solutions of the equation

\begin{equation}\label{eq:ms+1}
X^{s+1}={\alpha^{s+1}+1}
\end{equation}
for $\alpha$ ranging in $\GF{q^2}\setminus \{1,c_1,c_2\} $.

Assume by contradiction that Equation~\eqref{eq:ms+1} has at least one solution for every $\alpha \in \GF{q^2}\setminus \{1,c_1,c_2\} $. 
If follows that if  $\gamma$ is a solution of \eqref{eq:ms+1} for a given $\alpha$,  then  $c_i \gamma$ is a solution as well and the corresponding line $\ell_\alpha$ has at least three points in common with $\cF^2_{s+1}$. Moreover, the line with equations $X_0=X_2=X_4=0$ meets the Fermat curve in the points $(0,1,0,1,0)$, $(0,c_1,0,1,0)$, $(0,c_2,0,1,0)$.
This implies that the number $N$ of rational points of $\cF^2_{s+1}$ is $N\geq 3(q^2-2)+3=3(2^{2e}-1)$. On the other hand the Fermat curve is non-singular and hence,
 by the Hasse-Weil bound, $$N\leq q^2+1+qs(s-1)=2^{2e}+2^{2e-1}-2^{\frac{3e-1}{2}}+1.$$ 
Nevertheless, $2^{2e}+2^{2e-1}-2^{\frac{3e-1}{2}}+1\leq 3(2^{2e}-1)$ for $e \geq 3$. Therefore, this leads to a contradiction. Hence, it follows  that there is at least an element $\alpha$ such that the corresponding line $\ell_\alpha$ is external to $\cF^2_{s+1}$.

\subsection{Intersection of $\cV^4_{\varepsilon}$ with planes}

In the following, we show that $\tau_2(\cV^4_{\varepsilon})=0$
studying how  planes meet $\cV^4_{\varepsilon}$.

\begin{lemma}\label{teo:tangent_plane}
  In $\PG(4,q^2)$, there exists a plane that meets the variety $\cV^4_{\varepsilon}$ in exactly one point. 
\end{lemma}
\begin{proof}
  By the results of Subsection~\ref{subsec:ext_lines},
  there exists a line $\ell$ external to the Fermat curve $\cF^2_{s+1}$ in the plane $\Pi_0: X_0=X_4=0$. In order to prove the statement, it is enough to show that the plane $\Pi= \langle P_{\infty},\ell \rangle$ meets $\cV_{\varepsilon,\infty}$ in one point. By construction $P_{\infty}$ belongs to $\Pi \cap \cV^4_{\varepsilon,\infty}$; suppose that there exists a point $Q \in \Pi \cap \cV^4_{\varepsilon,\infty}$, $Q\neq P_{\infty}$. Since $\cV^4_{\varepsilon,\infty}$ is a cone with vertex $P_{\infty}$ over $\cF^2_{s+1}$ than there exists a point $R \in\cF^2_{s+1}$ such that $Q$ belongs to the line $P_{\infty}R$. Then $P_{\infty}Q=P_{\infty}R \subseteq \Pi$ and this meets the line $\ell$ in the point $R$, getting a contradiction.
\end{proof}

\begin{theorem}\label{teo:external_plane}
Each plane in  $\PG(4,q^2)$ meets the variety $\cV^4_{\varepsilon}$  in at least one point.
\end{theorem}
\begin{proof}
  Suppose there exists a plane $\Pi$ disjoint from  the variety $\cV^4_{\varepsilon}$. This plane must  meet the hyperplane at infinity in a line. Consider the $q^2+1$ hyperplanes of $\PG(4,q^2)$ containing $\Pi$.
By \Cref{th:inter_hyp} for $r=4$,
each of these planes meets  $\cV^4_{\varepsilon}\cap \AG(4,q^2)$ in $q^5-q^3+q^2 $, or $q^5$, or $ q^5+q^2$ affine points, 
Since the cardinality of $\cV^4_{\varepsilon}$ is $q^7$ it follows that
\[ q^7=|\cV^4_{\varepsilon}\cap\AG(4,q^2)|\geq(q^2+1)(q^5-q^3+q^2)> q^7,\]
a contradiction.
\end{proof}

\begin{lemma}
Each line  of $\PG(2,q^2)$ meets the  Fermat curve $\cF^2_{s+1}$  in  $0, 1, 2$ or $t$ points where $t=5$ if $e\equiv1\pmod4$ and $e>1$ or $t=3$ if $e\equiv3\pmod 4$.
\end{lemma}
\begin{proof}
  We need to solve  the following system in $X_1,X_2,X_3$:
\begin{equation} \label{intersection}
    \begin{cases}
    X_1^{s+1}+X_2^{s+1}+X_3^{s+1}=0\\
    \alpha X_1+\beta X_2+\gamma X_3=0.
    \end{cases}
\end{equation}
Since all the expressions are symmetric in $X_1, X_2, X_3$ we can
assume without loss of generality $\gamma\neq 0$ i.e., equivalently, set $\gamma=1$. 
Substituting the second equation of System~\eqref{intersection} in the first
we obtain
\begin{equation}\label{nic1}
(1+\alpha^{s+1})X_1^{s+1}+\alpha^s \beta X_1^{s}X_2+\alpha\beta^{s}X_1X_2^s+(1+\beta^{s+1})X_{2}^{s+1}=0.
\end{equation}
The number of solutions of Equation~\eqref{nic1} corresponds to
the number of absolute points in $\PG(1,2^{2e})$ of a (possibly
degenerate) $s$-sesquilinear form. By~\cite[Lemma 6]{HN} this
corresponds to either $0,1,2$ or $t$ points where $t$ is the
cardinality of a subline over the fixed subfield of $s$.
In particular, for $e\equiv1\pmod 4$ we have $t=5$ and for
$e\equiv3\pmod4$ we have $t=3$.
\end{proof}

%If $\alpha^{s+1}+\beta^{s+1}+1=0$, then the number of solutions of \eqref{nic1} coincides with the number of the absolute points of a degenerate $s$-sesquilinear form of $\PG{(1,2^{2e})}$ and the result follows from \cite{HN}.
%Thus, assume $\alpha^{s+1}+\beta^{s+1}+1\neq 0$.
%If $e\equiv3\pmod 4$ and $\alpha\notin\GF{2}$ or   $e\equiv 1\pmod 4$ and $\alpha\notin\GF{4}$,
%by the proof of Proposition \ref{Nic},  the number of solutions of \eqref{nic1} equals the number of solutions of an equation of the type:
%\begin{equation}\label{nic2}
%X^{s+1}+X+d=0
%\end{equation}
%Thus, the result follows from \cite{B}.
%
%If $\alpha=0$, then~\eqref{nic1} becomes
%\[ X_1^{s+1}+(1+\beta^{s+1})X_2^{s+1}=0.\]
%This has $1$ solution for $e\equiv1\pmod4$ and $0$ solutions
%for $e\equiv3\pmod4$.
%
%If $\alpha=1$ then necessarily $\beta\neq0$ and~\eqref{nic1} becomes
%\[ \beta X_1^sX_2+\beta^sX_1X_2^s+(1+\beta^{s+1})X_2^{s+1}=0.\]
%Clearly $X_2=0$ is a solution; so suppose $X_2\neq0$, divide by
%$X_2^{s+1}\beta^{s+1}$ and put $X=X_1/X_2$. Then we get
%\[ X^s+X+(1+\delta^{s+1})=0\]
%where $\delta=\beta^{-1}$.
%\dots
%
%Now we only  have to consider the case  $\alpha \in \GF{4} \setminus \GF{2}$ and  $e\equiv 1$ $\pmod 4$.
%In this case  $\alpha^s=\alpha$ and  Equation~\eqref{nic1} can be rewritten as Equation~\eqref{eq:t-u-v} with $t^s=u$.
%By performing  in \eqref{eq:t-u-v} the substitution $X\rightarrow X+t$, this becomes
%\[X^{s+1}+ut+v^{s+1}=0,\] which has exactly  one solution.
%\end{proof}

We can now specialize the results of \Cref{th:code4} to the case $r=4$.

\begin{proposition} \label{prop:code4}
The linear code $\cC^4_{\varepsilon}=\cC(\cV^4_{\varepsilon})$ for $q=2^e$ and $e>1$ an odd integer is a $5$-dimensional minimal code with the following parameters 
\begin{enumerate}
\item If $e \equiv 1 \pmod 4$, then $\cC^4_{\varepsilon}$ has length $|\cV^4_{\varepsilon}|=q^7+q^4+q^2+1$ and  weights
  \begin{multline*}
    \{ q^7, q^7-q^5+q^4+q^3-q^2, q^7-q^5+q^4+q^2, \\
    q^7-q^5+q^4,q^7-q^5+q^4-q^2,
    q^7-q^5+q^4-4q^2 \};
  \end{multline*}
\item If $e=3$, then $\cC^4_{\varepsilon}$ has length
$|\cV_{\varepsilon}^4|=q^7+q^4+2q^3+q^2+1$ and  weights  
\begin{multline*}
  \{ q^7, q^7-q^5+q^4+3q^3-q^2, q^7-q^5+q^4+2q^3+q^2, \\ 
  q^7-q^5+q^4+2q^3, q^7-q^5+q^4+2q^3-q^2,
  q^7-q^5+q^4+2q^3-2q^2\}.
\end{multline*}

%\begin{equation*}
%\{(q+1)^2+q^5+q^2,q^5+(q+1)^2+q^2,q^5-%q^3+q^2+1,q^5+q^2+1,q^5-q^3+1,q^5+1\}
%\end{equation*}
\end{enumerate}
\end{proposition}
\begin{proof} 
The result follows from \Cref{th:cone-over-Fermat}, \Cref{th:inter_hyp}, \Cref{thm:cardinalityinf} and \Cref{teo:tangent_plane}.  In order to compute the weights, it is enough to recall that 
$$ \vert \Pi_{\infty} \cap \cV_{\varepsilon}^4 \vert =1+q^2 \vert \cF_{n}^2 \vert $$
and
\begin{equation*}
     \vert \Pi \cap \cV_\varepsilon^4 \vert= \vert \Pi \cap \Pi_\infty \cap \cV_{\varepsilon}^4 \vert + \vert  \Pi \cap \cH^4_{\varepsilon} \vert - \vert \Pi \cap \Pi_{\infty} \cap \cH_{\varepsilon}^4 \vert 
\end{equation*}
for any hyperplane $\Pi$ distinct from $\Pi_{\infty}$.

\end{proof}
%\footnote{\textcolor{red}{Forse qualcuno di questi valori può essere eliminato esaminando la geometria di $H_{\varepsilon}^4$ e da lì abbiamo anche distanza minima.}}

Finally, by the results obtained above, we can state the following.

\begin{proposition}
    Consider the variety $\cV^4_{\varepsilon}$. Then 
        \begin{equation}
    \tau_1(\cV^4_{\varepsilon})=2, \quad   \tau_2(\cV^4_{\varepsilon})=2, \quad \tau_3(\cV^4_{\varepsilon})=0.       
    \end{equation}
\end{proposition}
\begin{proof}
The first  cutting gap follows by Subsection \ref{subsec:ext_lines}. The second one follows by
\Cref{teo:tangent_plane} and \Cref{teo:external_plane}.  Finally, by
\Cref{th:inter_hyp}, we have the third  
 cutting gap. 
\end{proof}
Note that since there are lines external to $\cV^4_{\varepsilon}$ we have
        $\tau_{1,0}(\cV^4_{\varepsilon})=1\neq\tau_1(\cV^4_{\varepsilon})$.

\vspace{0.5cm}
\noindent
\begin{minipage}{\textwidth}
Angela Aguglia, Viola Siconolfi\\
Dipartimento di Meccanica, Matematica e Management,\\
Politecnico di Bari,\\
Via Orabona, 4-70125 Bari, Italy\\
\texttt{\{angela.aguglia,viola.siconolfi\}@poliba.it}
\end{minipage}

\vspace{0.8cm}

\noindent
\begin{minipage}{\textwidth}
  Luca Giuzzi\\
D.I.C.A.T.A.M. \\
Università degli Studi di Brescia,\\
Via Branze, 43-25123 Brescia, Italy\\
\texttt{luca.giuzzi@unibs.it}
\end{minipage}

\vspace{0.8cm}

\noindent
\begin{minipage}{\textwidth}
Giovanni Longobardi\\
Dipartimento di Matematica ed Applicazioni 'R. Caccioppoli',\\
Università degli Studi di Napoli Federico II,\\
Via Cintia, Monte S.Angelo I-80126 Napoli, Italy\\
\texttt{giovanni.longobardi@unina.it}
\end{minipage}


\begin{thebibliography}{999}

	\bibitem{A}{\bf A.~Aguglia.} Quasi--Hermitian varieties in $\PG(r,q^2)$, $q$ even, \emph{Contrib. Discrete Math.} \textbf{8}(1) (2013), 31--37.


	\bibitem{ACG}{\bf A.~Aguglia, M.~Ceria, L.~Giuzzi.} Some hypersurfaces over finite fields, minimal codes and secret sharing schemes, \emph{Des. Codes Cryptogr.} \textbf{90} (2022), 1503--1519.
    
\bibitem{AK}{\bf A.~Aguglia, G. Korchmaros.} Multiple blocking sets and multisets in Desarguesian planes, \emph{Des. Codes Cryptogr.} \textbf{56} (2010), 177--181

\bibitem{ABN}{\bf G.N.~Alfarano, M.~Borello, A.~Neri.}
A geometric characterization of minimal codes and their asymptotic performance, \emph{Adv. in Math. of Communications} \textbf{16} (2022), 115--133.

\bibitem{AB}{\bf A.~Ashikhmin, A.~Barg.}
Minimal vectors in linear codes, \emph{IEEE Trans. Inf. Theory} \textbf{14} (1998), 2010--2017.



\bibitem{B}{\bf A.W.~Bluher.}
On $x^{q+1}+ax+b$, \emph{Finite Fields and Their Applications }\textbf{10} (2004), 285--305.

\bibitem{BB}
{\bf M. Bonini, M. Borello}, Minimal linear codes arising from blocking sets, \emph{J. Algebraic Combin. } \textbf{53} (2021), 327--341.

\bibitem{Burton} 
{\bf R.C. Burton}, An application of convex sets to the construction of error correcting codes and factorial designs. Insitute of Statistics Mimeo Series 393, \emph{University of North Carolina, Chapel Hill} (1964).

\bibitem{CG}{\bf A.R. Calderbank, J.M. Goethals.}
Three-weight codes and association schemes
\emph{Philips J. Res.}, \textbf{39} (1984), 143--152.



\bibitem{CK}{\bf A. R. Calderbank, W.M. Kantor.}
The geometry of two-weight codes, \emph{Bull. Lond. Math. Soc. }\textbf{18} (1986), 97--122.

\bibitem{violating}{\bf H. Chen, Y. Chen, C. Xie, H. Lao.}
Minimal cinear Codes violating the Ashikhmin-Barg condition from arbitrary projective linear codes, \href{https://arxiv.org/abs/2505.07130}{https://arxiv.org/abs/2505.07130} (2025).


\bibitem{CMP}
{\bf G.D. Cohen, S. Mesnager, A. Patey.} On minimal and quasi-minimal linear codes \emph{M.Stam (Ed.) IMACC 2013}, LNCS,  vol. \textbf{8308}, Springer, Heidelberg, (2013), 85--98

\bibitem{DGMP}
  {\bf A.A. Davydov, M. Giulietti, S. Marcugini, F. Pambianco}
  Linear nonbinary covering codes and saturating sets in projective spaces
  \emph{Adv. Math. Commun.} \textbf{5}:1 (2011), 119--147.

\bibitem{HN} {\bf J. d'Haeseleer, N. Durante.} On absolute points of
correlations in $\PG(2,q^n)$, 
\emph{Elect. J. Combin.} {\bfseries 27(2)}, \#P2.32 (2020).


\bibitem{HKT}{\bf J.P.W.~Hischfeld, G.~Korchmaros, F.~Torres.}
Algebraic Curves over a Finite Field, \emph{Princenton series in applied Mathematics, Princeton University Press} (2008) .

\bibitem{HT} {\bf J.P.W. Hirschfeld, J.A. Thas.} General Galois Geometries , {\em Springer London, London}, 2016.





%\bibitem{K71} {\bf N.M. Katz}. On a theorem of Ax,
%  \emph{Amer. J. Math.} \textbf{93(2)} (1971), 485--499.

\bibitem{HK} {\bf M. Homma , S.J.  Kim.} An elementary bound for the number of points of a hypersurface over a finite field
\emph{
Finite Fields Appl.} \textbf{20} (2013), 76–-83.



\bibitem{HP}{\bf W.C.~Huffman, V.~Pless.}
Fundamentals of Error-Correcting Codes, \emph{Cambridge University Press,
Cambridge}, (2003).

\bibitem{MacWilliams} {\bf F.J. MacWilliams},
Combinatorial problems of elementary abelian groups,
\emph{Redcliffe College, Cambridge} (1962).

\bibitem{Massey} {\bf J.L. Massey.} Minimal codewords and secret sharing, \textit{Proceedings of the 6th Joint Swedish-Russian
International Workshop on Information Theory} (1993), 276--279.




\bibitem{S22} 
{\bf J. Schillewaert and J. Van de Voorde} Quasi‐Polar Spaces, \textit{Comb. Theory}  \textbf{2}(3) (2022), 2022-10.


\bibitem{TQLZ}{\bf C.~Tang, Y.~Qiu, Y.~Liao, Z.~Zhou.}
Full characterization of minimal
linear codes as cutting blocking sets, \emph{IEEE Trans. on Information
Theory} \textbf{67} (2021), 3690–-3700.



\bibitem{TsVla}{\bf M.A.~Tsfasman, S.G.~Vlăduţ.} Algebraic-Geometric Codes, Mathematics and Its Applications (Soviet
Series) {\bf 58,} \emph{Kluwer Academic Publishers Group, Dordrecht}, (1991).

\bibitem{W91} {\bf V.K. Wei.}
Generalized Hamming Weights for Linear Codes,
\emph{IEEE Trans. on Information Theory} \textbf{37}(5) (1991), 1412--1418.




\end{thebibliography}
\end{document}